\RequirePackage{fix-cm}
\documentclass[smallextended]{svjour3}       
\smartqed  

\usepackage[dvips]{graphicx}

\usepackage{amsmath,latexsym}
\usepackage{amsfonts}

\def\Ker{\mathop{\rm Ker}}
\def\dim{\mathop{\rm dim}}

\newcommand{\bR}{\mathbb{R}}

\newcommand{\cX}{{\cal X}}
\newcommand{\cY}{{\cal Y}}

\def\cY{{\cal Y}}

\def\rank{{\rm rank}}
\def\Im{{\rm Im}}

\def\Label{\label}

\newenvironment{proofof}[1]{\vspace*{5mm} \par 
\noindent{\it Proof of #1:\hspace{2mm}}}{\endproof
\hfill$\Box$ 
}

\def\PF#1{\noindent{\sf #1}:\quad}

\begin{document}

\title{Local Equivalence Problem in Hidden Markov Model}


\author{Masahito Hayashi}


\institute{M. Hayashi \at
              the Graduate School of Mathematics, Nagoya University, Japan,
              Center for Advanced Intelligence Project, RIKEN, Japan,
              Shenzhen Institute for Quantum Science and Engineering, Southern University of Science and Technology, China,
               \& 
               the Centre for Quantum Technologies, National University of Singapore, Singapore\\
              \email{e-mail:masahito@math.nagoya-u.ac.jp}           
}

\date{Received: date / Accepted: date}

\maketitle

\begin{abstract}
In the hidden Markov process, there is a possibility that
two different transition matrices for hidden and observed variables
yield the same stochastic behavior for the observed variables.
Since such two transition  matrices cannot be distinguished, we need to identify them
and consider that they are equivalent, in practice.
We address the equivalence problem of hidden Markov process
in a local neighborhood by using the geometrical structure of hidden Markov process.
For this aim, we introduce
a mathematical concept to express Markov process,
and formulate its exponential family by using generators.
Then, the above equivalence problem is formulated as the equivalence problem of generators.
Taking this equivalence problem into account, we derive several concrete parametrizations 
in several natural cases.
\end{abstract}

\section{Introduction}\Label{s1}
An observed variable $Y$ subject to a hidden Markov process
is  determined by a hidden variable $X$ subject to Markov process.
That is, 
the stochastic behavior of variables subject to a hidden Markov process 
is characterized by a pair of a transition matrix $V$ from the hidden variable $X$ to the observed variable $Y$
and a transition matrix $W$ on the hidden variable $X$ as Fig. \ref{2model}.
Now, we assume that our interest is limited to the stochastic behavior of the observed variable $Y$,
which is described by the pair of transition matrices.
However, there is ambiguity for the pair of the function and the transition matrix 
to express the hidden Markov process
when our interest is limited to the stochastic behavior of the observed variable $Y$. 
That is, there is a possibility that two different pairs express 
the same stochastic behavior of the observed variable $Y$. 
The problem to characterize such two different pairs is called the equivalence problem. 
When the transition matrix $V$ is given by a deterministic function $f$ as Fig. \ref{1model},
it was solved by Ito, Amari and Kobayashi \cite{IAK}.
When the number of states in the hidden system cannot be identified,
we need to choose the minimum number of the states to realize 
the given stochastic behavior of the observed variable $Y$. 
This kind of problem might be crucial
when the structure of hidden Markov process is not known.
However, since the asymptotic error is characterized by the local geometrical structure,
to discuss the estimation of the hidden Markov process, we need to consider this problem in the tangent space, which was not addressed in \cite{IAK}.
Indeed, as explained later, this problem is deeply related to the geometrical structure of hidden Markov process.

As another problem, 
we address the formulation of information geometrical structure, 
especially, an exponential family, for hidden Markov process.
Indeed, information geometry was established by Amari and Nagaoka \cite{AN}
as a very powerful method for statistical inference.
Nakagawa and Kanaya \cite{NK} and Nagaoka \cite{HN} 
addressed its extension to Markov process and 
formulated an exponential family for transition matrices.
As an advantage of an exponential family for transition matrices,
the geometric structure depends only on the transition matrices, and
it does not change as the number $n$ of observation increases
while the geometry based on the probability distribution changes according to the increase of the number $n$. 
Recently, the paper \cite{HW-est} applied this geometrical structure to estimation of Markov process, and clarified the importance of this kind of exponential families for statistical inference
by employing the following two facts;
Information geometry of an exponential family for transition matrices is given as 
Bregman divergence \cite{Br,Am} of the cumulant generating function $\phi(\vec{\theta})$.
All the asymptotic statistical properties can be recovered by 
the cumulant generating function $\phi(\vec{\theta})$ in the Markov process \cite{HW14-2}.
In particular, when the unknown transition matrix is assumed to belong to an exponential family for transition matrices, the asymptotic efficiency of the estimator for the expectation parameter
was shown in the same way as the independent and identical distribution \cite{HW-est}.
However, the formulation of exponential family for hidden Markov process was not discussed in these existing papers.
This formulation is needed when we extend the idea in  \cite{HW14-2} to the hidden Markov process \cite{HHM}.

In this paper, to formulate an exponential family for hidden Markov process, 
due to the following reason,
we address the model given in Fig \ref{3model} for hidden Markov process, in which,
the next hidden variable and the observed variable are correlated even when the previous hidden variable is fixed.
Indeed, the model of Fig. \ref{1model} is generalized to 
the model of Fig. \ref{2model} by replacing the deterministic function $f$ by another transition matrix $V$. 
Both models have a complicated structure to define an exponential family directly.
At least, when we employ these models, the definition of an exponential family is not so natural.
In contrast, as explained in Remark \ref{R2}, 
the model given in Fig \ref{3model} is most convenient for the discussion of the equivalence problem,
and contains the above two cases.
Notice that by extending the hidden system,  
the model of Fig. \ref{1model} includes the model of Fig. \ref{3model},
which shows the equivalence among three models.
Hence, we formulate the model of Fig. \ref{3model} by introducing a mathematical concept 
{\it ${\cal Y}$-indexed transition matrix}, and define 
an exponential family of ${\cal Y}$-indexed transition matrices.
In this definition, generators play an essential role and express the infinitesimal changes.
The local equivalence problem is reduced to the equivalent problem for generators.
That is, we derive a necessary and sufficient condition for an infinitesimal change of the transition matrix 
to be distinguished. 
In this way, we can discuss the above two tasks simultaneously.

Further, we address several concrete examples.
For example, we give a concrete parametrization taking the local equivalence into account
when the hidden system and the observed system are composed of two states.
Also, we apply the definition of an exponential family of ${\cal Y}$-indexed transition matrices
to the model given in Fig. \ref{2model}.
Then, we characterize the local equivalence in this special case more concretely.
In particular, under a certain natural condition, 
we give a concrete parametrization under this model.

\begin{figure}[htbp]
\begin{center}
\scalebox{1}{\includegraphics[scale=0.5]{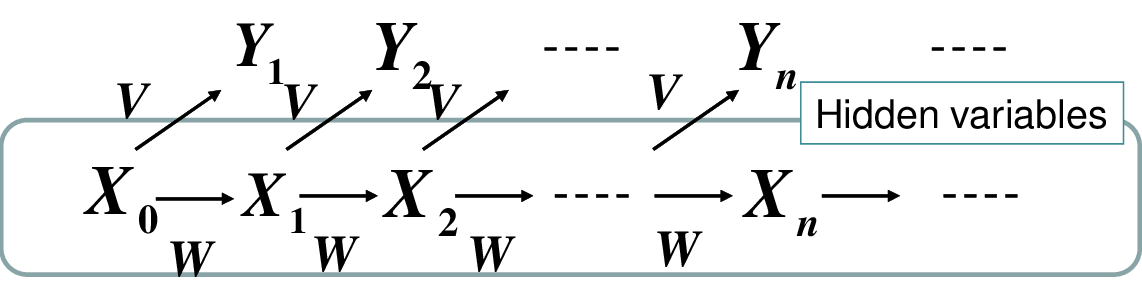}}
\end{center}
\caption{The second model: 
The transition matrix $W$ determines the Markov process on the set ${\cal X}$ of hidden states.
The transition matrix $V$ determines the observed variable $Y$
with the condition on the hidden variable $X$.}
\Label{2model}
\end{figure}%

\begin{figure}[htbp]
\begin{center}
\scalebox{1}{\includegraphics[scale=0.5]{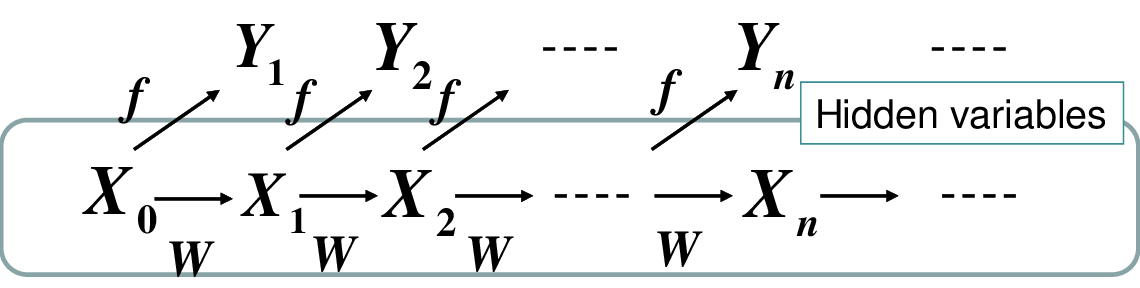}}
\end{center}
\caption{The first model: 
The transition matrix $W$ determines the Markov process on the set ${\cal X}$ of hidden states.
The function $f$ of the hidden variable $X$ determines the observed variable $Y$.}
\Label{1model}
\end{figure}%

\begin{figure}[htbp]
\begin{center}
\scalebox{1}{\includegraphics[scale=0.5]{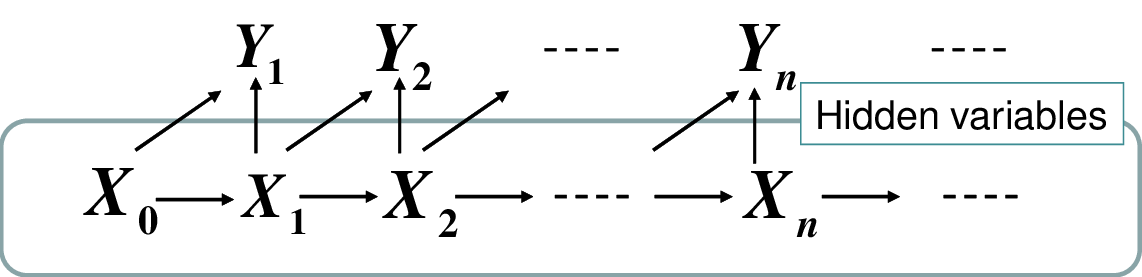}}
\end{center}
\caption{The third model: 
the hidden variable $X_i$ and the observed variable $Y_i$ are correlated even when 
the previous hidden variable $X_{i-1} $is fixed.}
\Label{3model}
\end{figure}%

The remaining of this paper is organized as follows.
Section \ref{s2} gives a brief summary of the obtained results, which is crucial for understanding the structure of this paper.
Section \ref{s5} introduces the notion of ${\cal Y}$-indexed transition matrix
to describe the model given in Fig. \ref{3model},
and revisits the equivalence problem of hidden Markov process under this formulation.
Section \ref{s6} introduces an exponential family of  ${\cal Y}$-indexed transition matrices,
and discusses the local equivalence problem under the model given in Fig. \ref{3model}.
By taking local equivalence into account,
The remaining sections are outlined in Subsection \ref{s2-3}.

\section{Summary of obtained results}\Label{s2}
\subsection{Global equivalence}
First, we adopt the model given in Fig. \ref{3model}, which is called the general model.
To address this model,
we consider a collection of non-negative matrices $\vec{W}= (W_y(x|x'))_{y \in {\cal Y}}$ on the hidden system ${\cal X}$
with the condition that $\sum_{y \in {\cal Y}}W_y$ is a probability transition matrix,
where
a matrix is called non-negative when all of its matrix components are non-negative.
Then, we have the transition matrix 
$P_{Y,X|X'}(y,x|x'):=W_y(x|x')$, which describes the stochastic behavior of this model.
When the initial distribution is given by a distribution $P$ on ${\cal X}$,
we have the joint distribution of the observed sequence $Y_k,\ldots, Y_1$ as
$P^{k}[\vec{W}] \cdot P (y_k,\ldots, y_1):=
\sum_{x\in {\cal X}}P^k[\vec{W}](y_k, \ldots, y_1|x)P(x)$,
where
the transition matrix $P^k[\vec{W}]$ from ${\cal X}$ to ${\cal Y}^k$ is given as
\begin{align}
P^k[\vec{W}](y_k, \ldots, y_1|x'):=\sum_{x\in {\cal X}}W_{y_k}\cdots W_{y_1}(x|x').
\end{align}
Then, a pair of $\vec{W}$ and $P$ is called equivalent to another pair of $\vec{W}'$ and $P'$
when they have the same stochastic behavior with respect to 
of the observed sequence $Y_k,\ldots, Y_1$ with an arbitrary $k$.
We call $\vec{W}$ a ${\cal Y}$-indexed transition matrix on ${\cal X}$.
When the transition matrix
$|\vec{W}|(x|x'):=\sum_{y \in {\cal Y}} W_y(x|x')$ on ${\cal X}$ is irreducible,
the ${\cal Y}$-indexed transition matrix $\vec{W}$ is called irreducible.
For the following discussion, we employ the vector space 
${\cal V}_{{\cal X}}:= \{ v=(v_x)_{x \in {\cal X}}| v_x \in \bR\}$, 
i.e., the space ${\cal V}_{{\cal X}}$ is spanned by basis $\{e_x\}_{x \in {\cal X}}$.
The matrix $P^k[\vec{W}]$ can be regarded as a linear map from 
${\cal V}_{{\cal X}}$ to ${\cal V}_{{\cal Y}^k}$.
Define $k_{\vec{W}}$ as the minimum integer 
to satisfy the condition $\Ker P^k[\vec{W}]=P^n[\vec{W}]$ for $n \ge k$.

In the following, we regard a distribution $P$ on ${\cal X}$ as an element of $ {\cal V}_{{\cal X}}$.
Then, $[P]$ denotes an element of
the quotient space $ {\cal V}_{{\cal X}}/\Ker P^{k_{\vec{W}}}[\vec{W}]$
whose representative is $P$.
Given a distribution $P$ on $\cX$ and a positive integer $k$, 
we define the subspace ${\cal V}^k(P)$ of 
${\cal V}_{{\cal X}}/ \Ker P^{k_{\vec{W}}}[\vec{W}]$
spanned by $\{ [W_{y_{k'}}\cdots W_{y_1}P] | y_j \in \cY, k'\le k\}$.
We denote the minimum integer $k_0$ satisfying the following condition by $k_{(P,\vec{W})}$,
and call it the {\it minimum length} of $(P,\vec{W})$:
\begin{align}
\cup_{k =1}^{\infty} {\cal V}^k(P) = {\cal V}^{k_0}(P),
\end{align}
where the existence of the minimum will be shown later (Lemma \ref{L2T}).
Then, we have the following theorem.

\begin{theorem}\Label{sum-L-9-27}
The following conditions for 
two collections of non-negative matrices $\vec{W}$, $\vec{W}'$ 
and
two distributions $P$, $P'$ on $\cX$ 
are equivalent.
\begin{description}
\item[\bf (A1)]
There exists an invertible map $T$
from ${\cal V}^{k_{(P,\vec{W})}}(P)$
to ${\cal V}^{k_{(P',\vec{W}')}}(P')$ such that
the relation $T [W_y]  =[W_y'] T$ 
holds for $y\in {\cal Y}$ 
and the equation $T [P]=[P']$ holds.

\item[\bf (A1)']
{\bf (A1)} holds, and 
the relations $k_{\vec{W}}=k_{\vec{W}'}$,
$k_{(P,\vec{W})}=k_{(P',\vec{W}')}$, and $d_{(P,\vec{W})}=d_{(P',\vec{W}')}$
hold.

\item[\bf (A2)]
The pair of $\vec{W}$ and $P$ is equivalent to the pair of $\vec{W}'$ and $P'$.

\item[\bf (A3)]
The relations 
$k_{\vec{W}}=k_{\vec{W}'}$,
$k_{(P,\vec{W})}=k_{(P',\vec{W}')}$,
and
$P^{k_{\vec{W}}+k_{(P,\vec{W})}+1}[\vec{W}] \cdot P
=P^{k_{\vec{W}}+k_{(P',\vec{W}')}+1}[\vec{W}'] \cdot P'$ 
hold.

\item[\bf (A4)]
The relation 
$P^{k}[\vec{W}] \cdot P
=P^{k}[\vec{W}'] \cdot P'$ 
holds for $k=\max (k_{\vec{W}},k_{\vec{W}'})
+\max(k_{(P,\vec{W})},k_{(P',\vec{W}')})+1$.
\end{description}
\end{theorem}
While the equivalence between {\bf (A1)} and {\bf (A2)} was shown in \cite{IAK},
other equivalence relations were not shown.

Due to this theorem,
in order to check the equivalence condition, 
it is sufficient to check the stochastic behavior 
of the observed sequence $Y_k,\ldots, Y_1$ with length $k=\max (k_{\vec{W}},k_{\vec{W}'})
+\max(k_{(P,\vec{W})},k_{(P',\vec{W}')})+1$.

\subsection{Local equivalence}
Given an irreducible ${\cal Y}$-indexed transition matrix $\vec{W}$, 
we consider the set ${\cal G}_1(({\cal Y}\times {\cal X}^{2})_{\vec{W}})$
of functions $g(y,x,x')$ to satisfy the condition 
$\sum_{y\in {\cal Y},x\in {\cal X}}g(y,x,x')W_{y}(x|x')=0$ for any $x'\in {\cal X}$.
Given functions $g_1, \ldots, g_l \in {\cal G}_1(({\cal Y}\times {\cal X}^{2})_{\vec{W}})$,
we consider the exponential family $\{\vec{W}_{\vec{\theta}}\}_{\vec{\theta}}$ of 
${\cal Y}$-indexed transition matrices on ${\cal X}$
generated by the generators $\{g_i\}_{i=1}^l$, which is defined in Subsection \ref{s6-1}.

Then, as the derivative version of the conditions {\bf (A3)} and {\bf (A4)} 
with the parametrization $\vec{\theta}=(a^j t)_{j=1}^l$ for
a vector $\vec{a} \in \mathbb{R}^l$, 
we have the following conditions.
\begin{description}
\item[\bf (B3)]
The relation $\sum_{j=1}^l a^j \frac{\partial }{\partial \theta^j}
P^{k_{\vec{W}}+k_{(P,\vec{W})}+1}[\vec{W}_{\vec{\theta}}]\cdot P \Big|_{\vec{\theta}=0}= 0$ holds.

\item[\bf (B4)']
The relation $\sum_{j=1}^l a^j \frac{\partial }{\partial \theta^j}
P^{k}[\vec{W}_{\vec{\theta}}]\cdot P \Big|_{\vec{\theta}=0}= 0$ holds 
for any integer $k \ge k_{(W,V)}+k_{(P_{(W,V)},(W,V))}+1$.
\end{description}
These conditions can be considered as 
equivalence condition

To characterize these conditions,  
we also consider subspaces 
${\cal N}_P(({\cal Y}\times {\cal X}^{2})_{\vec{W}})$,
${\cal N}_{2,P}(({\cal Y}\times {\cal X}^{2})_{\vec{W}})$, and
${\cal N}_2(({\cal Y}\times {\cal X}^{2})_{\vec{W}})$
of 
${\cal G}_1(({\cal Y}\times {\cal X}^{2})_{\vec{W}})$ as defined in Subsection \ref{s6-1}.
Then, we have the following theorem.

\begin{theorem}\Label{sum-L27-2}
Given 
a distribution $P$ on ${\cal X}$,
the following conditions are equivalent with the conditions {\bf (B3)} and {\bf (B4)'} 
for 
a vector $\vec{a} \in \mathbb{R}^l$.

\begin{description}
\item[\bf (B1)]
The function $\sum_{j=1}^l a^j g_j \in {\cal G}_1(({\cal Y}\times {\cal X}^{2})_{\vec{W}})$
belongs to 
$
{\cal N}_P(({\cal Y}\times {\cal X}^{2})_{\vec{W}})
+{\cal N}_{2,P}(({\cal Y}\times {\cal X}^{2})_{\vec{W}})$.

\end{description}
\end{theorem}

Next, we focus on the stationary distribution $P_{\vec{W}_{\vec{\theta}}}$ of the transitive matrix
$|\vec{W}_{\vec{\theta}}|$.
Then, we have the following theorem.

\begin{theorem}\Label{sum-27-3}
The following conditions are equivalent for 
a vector $\vec{a} \in \mathbb{R}^l$.
\begin{description}
\item[\bf (C1)]
The function $\sum_{j=1}^l a^j g_j \in {\cal G}_{1}(({\cal Y}\times {\cal X}^{2})_{\vec{W}})$
belongs to $
{\cal N}_2(({\cal Y}\times {\cal X}^{2})_{\vec{W}})
+{\cal N}_{P_{\vec{W}}}(({\cal Y}\times {\cal X}^{2})_{\vec{W}})$.


\item[\bf (C3)]
The relation $\sum_{j=1}^l a^j \frac{\partial }{\partial \theta^j}
P^{k_{\vec{W}}+k_{(P_{\vec{W}},\vec{W})}+1}
[\vec{W}_{\vec{\theta}}]
\cdot P_{\vec{W}_{\vec{\theta}}} \Big|_{\vec{\theta}=0} = 0$ holds.

\item[\bf (C4)']
The relation $\sum_{j=1}^l a^j \frac{\partial }{\partial \theta^j}
P^{k}[\vec{W}_{\vec{\theta}}]
\cdot P_{\vec{W}_{\vec{\theta}}} \Big|_{\vec{\theta}=0} = 0$ holds
with any integer $k \ge k_{(W,V)}+k_{(P_{(W,V)},(W,V))}+1$.
\end{description}
\end{theorem}

These theorems clarify the following points.
In order check the zero-derivative condition, 
it is sufficient to check the stochastic behavior 
of the observed sequence $Y_k,\ldots, Y_1$ with length $k=\max (k_{\vec{W}},k_{\vec{W}'})
+\max(k_{(P,\vec{W})},k_{(P',\vec{W}')})+1$.
Also, the zero-derivative conditions can be converted to 
the conditions {\bf (B1)} and {\bf (C1)}
by using the subspaces
${\cal N}_P(({\cal Y}\times {\cal X}^{2})_{\vec{W}})$,
${\cal N}_{2,P}(({\cal Y}\times {\cal X}^{2})_{\vec{W}})$, 
${\cal N}_2(({\cal Y}\times {\cal X}^{2})_{\vec{W}})$, and
${\cal N}_{P_{\vec{W}}}(({\cal Y}\times {\cal X}^{2})_{\vec{W}})$
of 
${\cal G}_1(({\cal Y}\times {\cal X}^{2})_{\vec{W}})$.

\subsection{Outline of remaining parts}\Label{s2-3}
Next, we outline the results of remaining parts (Sections \ref{s6-3}-\ref{s-Fi}).
Section \ref{s8} considers the case when 
$Y_n$ and $X_n$ are conditionally independent with a fixed value $X_{n-1}=x_{n-1}$, as illustrated in 
Fig. \ref{2model}.
We call this model the conditionally independent model
while the model with Fig \ref{3model} is called the general model.
In this case,
the subspaces ${\cal N}_P(({\cal Y}\times {\cal X}^{2})_{\vec{W}})$
and
${\cal N}_{2,P}(({\cal Y}\times {\cal X}^{2})_{\vec{W}})$
can be simplified in a simpler way as Theorem \ref{FET2} in Section \ref{s8}.

We give several concrete constructions of generators
for the general model in Section \ref{s6-3} 
and 
for the conditionally independent model in Example \ref{ERT} in Section \ref{s8}.
More precise analysis for the two-hidden-state case
is done 
for the general model in Section \ref{s7} 
and 
for the conditionally independent model in Section \ref{s-Fi}.

\section{Hidden Markov model and equivalence}\Label{s5}
\subsection{Notations with ${\cal Y}$-indexed transition matrix}\Label{s5-1}
In the hidden Markov process, there is a possibility that
two different transition  matrices for hidden and observed variables
yield the same stochastic behavior of the observed variables.
Since such two transition  matrices cannot be distinguished, we need to identify them
and consider that they are equivalent, in practice.
In this section, we discuss the equivalence problem of hidden Markov process.
This subsection prepares notation for this aim.

Usually, a hidden Markov process is given as the combination of a Markov chain on a hidden finite state system ${\cal X}$
and a function of the hidden system ${\cal X}$ to a visible finite state system ${\cal Y}$ like Fig. \ref{1model}.
The paper \cite{IAK} discusses the equivalence problem of hidden Markov process in this formalism.
However, it requires a very complicated notation because it does not directly treat the set of observed values.
To avoid this problem, in this paper, we treat a hidden Markov process in a different form.
That is, we consider a collection of non-negative matrices $\vec{W}= (W_y(x|x'))_{y \in {\cal Y}}$ on the hidden system ${\cal X}$
with the condition that $\sum_{y \in {\cal Y}}W_y$ is a probability transition matrix,
where
a matrix is called non-negative when all of its matrix components are non-negative.
In this formulation, when the input is $x'$, we observe the visible outcome $y$ with probability
$\sum_{x \in {\cal X}} W_y(x|x')$.
This formalism directly expresses the behavior of observed outcomes so that the equivalence problem can be easily addressed.
Under this observation $Y=y$, the resultant distribution $P_{X|YX'}$ on ${\cal X}$
is $P_{X|YX'}(x|yx')=W_y(x|x')/\sum_{\tilde{x} \in {\cal X}} W_y(\tilde{x}|x') $.
Since the observed outcome takes values in the system ${\cal Y}$,
we call $\vec{W}$ a ${\cal Y}$-indexed transition matrix on ${\cal X}$\footnote{A ${\cal Y}$-indexed transition matrix on ${\cal X}$ can be regarded as the classical version of
measuring instrument of the quantum setting \cite{H2nd,Ozawa}, which describes the quantum measuring process.
The recent paper \cite{HY} characterizes quantum hidden Markov process by using measuring instrument.}.
When the initial distribution $P_{X_0}$ is given, 
like Fig. \ref{3model}, we have the joint distribution of the sequence 
$X_n,Y_n,X_{n-1},Y_{n-1} \ldots, X_1,Y_1,X_0$ as
\begin{align}
&P_{X_n,Y_n,X_{n-1},Y_{n-1} \ldots, X_1,Y_1,X_0}
(x_n,y_n,x_{n-1},y_{n-1} \ldots, x_1,y_1,x_0)\nonumber \\
:=&
W_{y_n}(x_n|x_{n-1})
W_{y_{n-1}}(x_{n-1}|x_{n-2}) \cdots
W_{y_1}(x_1|x_0)
P_{X_0}(x_0).
\end{align}
That is, 
$X_n$ and $Y_n$ are correlated even when $X_{n-1}$ is fixed to a value $x_{n-1}$.

When we are given a Markov process $W$ on ${\cal X}$ and a function $f:{\cal X} \to {\cal Y}$ as the conventional formalism of hidden Markov process, 
we have a disjoint partition $({\cal X}_y)_{y \in {\cal Y}}$ of ${\cal X}$ by defining ${\cal X}_y:= f^{-1}(y)$.
When we define the collection $\vec{W}= (W_y(x|x'))_{y \in {\cal Y}}$ as 
\begin{align}
W_y(x|x')=
\left\{
\begin{array}{ll}
0 & \hbox{when } x \notin  {\cal X}_y \\
W(x|x') & \hbox{when } x \in  {\cal X}_y,
\end{array}
\right.
\end{align}
the collection 
$\vec{W}$
gives a hidden Markov process under our formalism.
If  the function $f$ is one-to-one, 
$Y$ is subject to Markov process.

Conversely, once a collection $\vec{W}= (W_y(x|x'))_{y \in {\cal Y}}$ is given,
we have a hidden Markov process $W$ on ${\cal X}'$ and a function $f:{\cal X}' \to {\cal Y}$ as follows.
Define the set $\tilde{{\cal X}}:= {\cal X} \times {\cal Y}$ and the map $f$ as $f(x,y):=y$.
Then, we can define the transition matrix 
$\vec{W}_{|\tilde{{\cal X}}}$
on $\tilde{{\cal X}}={\cal X}\times {\cal Y}$ by 
\begin{align}
\vec{W}_{|\tilde{{\cal X}}}(x,y|x',y'):=W_y(x|x')
\Label{LGR},
\end{align}
which yields the joint Markov process.
The pair of $\vec{W}_{\tilde{{\cal X}}}$ and the function $f$ 
recovers the conventional formalism of hidden Markov process.
In this way, our formalism and the conventional formalism can be converted to each other.

Given a ${\cal Y}$-indexed transition matrix 
$\vec{W}= (W_y(x|x'))_{y \in {\cal Y}}$ on ${\cal X}$,
we denote the transition matrix $\sum_{y \in {\cal Y}} W_y$ on ${\cal X}$
by $|\vec{W}|$.
A ${\cal Y}$-indexed transition matrix $\vec{W}$ is called irreducible
when $|\vec{W}|$ is irreducible.
In this case, the average $ \sum_{i=1}^n \frac{1}{n} |\vec{W}|^i P$
converges to the stationary distribution $P_{\vec{W}}$ for any initial distribution $P$
as $n$ goes to infinity \cite{DZ,kemeny-snell-book}.
In the following,
for simplicity, we identify ${\cal X}$ and ${\cal Y}$
with $\{1,\ldots, d\}$ and $\{1,\ldots, d_Y\}$, respectively.
That is, $|{\cal X}|=d$ and $|{\cal Y}|=d_Y$.
Also, we assume that a ${\cal Y}$-indexed transition matrix $\vec{W}$ is irreducible.
Even in this assumption, $\vec{W}_{|\tilde{{\cal X}}}$ is not necessarily irreducible.
Hence, the distribution $P_{\vec{W}_{|\tilde{{\cal X}}}}$ is not uniquely defined.
However, when we define it as
\begin{align}
P_{\vec{W}_{|\tilde{{\cal X}}}}(x,y)
:=\sum_{x'\in {\cal X}} W_{y}(x|x') P_{\vec{W}}(x'),\Label{Eq3-7-1}
\end{align}
we have the following lemma.

\begin{lemma}
The distribution $P_{\vec{W}_{|\tilde{{\cal X}}}}$
is an eigenvector of $\vec{W}_{|\tilde{{\cal X}}}$,
i.e., an invariant distribution on the product space ${\cal X}\times {\cal Y} $
under the transition matrix $\vec{W}_{|\tilde{{\cal X}}}$.
\end{lemma}

\begin{proof}
The desired statement can be shown in the following way.
\begin{align}
&(\vec{W}_{|\tilde{{\cal X}}} P_{\vec{W}_{|\tilde{{\cal X}}}})(x,y)
=\sum_{x',y'} W_{y}(x|x') \sum_{x''} W_{y'}(x'|x'') P_{\vec{W}}(x'') \nonumber \\
=&\sum_{x'} W_{y}(x|x') \sum_{x'',y'} W_{y'}(x'|x'') P_{\vec{W}}(x'')
=\sum_{x'} W_{y}(x|x') P_{\vec{W}}(x').
\end{align}
\hfill$\Box$\end{proof}

Now, we discuss the equivalence relation for ${\cal Y}$-indexed transition matrices on ${\cal X}$.
When we focus on $k$ values $(y_k,y_{k-1}, \ldots, y_1)$ on ${\cal Y}$ subject to the process described by 
the ${\cal Y}$-indexed transition matrix $\vec{W}= (W_y(x|x'))_{y \in {\cal Y}}$ on ${\cal X}$,
the joint stochastic behavior of $(y_k,y_{k-1}, \ldots, y_1)$ and the input and output values in ${\cal X}$
is described by the ${\cal Y}^k$-indexed transition matrix 
$\vec{W}^{(k)}:=
(W_{y_k}\cdot W_{y_{k-1}}\cdots W_{y_1} (x|x'))_{(y_k,y_{k-1}, \ldots, y_1) \in {\cal Y}^k}$ on ${\cal X}$.
That is, we observe $k$ outcomes in ${\cal Y}$ subject to the transition matrix
\begin{align}
P^k[\vec{W}](y_k, \ldots, y_1|x'):=\sum_{x\in {\cal X}}W_{y_k}\cdots W_{y_1}(x|x')\Label{LEV}
\end{align}
and the initial distribution on ${\cal X}$.
For the following discussion, we employ the vector space 
${\cal V}_{{\cal X}}:= \{ v=(v_x)_{x \in {\cal X}}| v_x \in \bR\}$, 
i.e., the space ${\cal V}_{{\cal X}}$ is spanned by basis $\{e_x\}_{x \in {\cal X}}$.
Then, the transition matrix $P^k[\vec{W}] $ can be regarded as a linear map from
${\cal V}_{{\cal X}}$ to ${\cal V}_{{\cal Y}^k}={\cal V}_{{\cal Y}}^{\otimes k}$.
So, we have $\Ker P^k[\vec{W}] \subset \Ker P^{k_0}[\vec{W}]$ for $k \ge k_0$.
We denote the minimum integer $k_0$ satisfying the following condition by $k_{\vec{W}}$,
and call it the {\it minimum length} of $\vec{W}$:
\begin{align}
\cap_{k =1}^{\infty} \Ker P^k[\vec{W}]= \Ker P^{k_0}[\vec{W}],
\end{align}
where the existence of the minimum is shown in Lemma \ref{L1T}.
The dimension $d_{\vec{W}}:= \dim ({\cal V}_{{\cal X}}/ \Ker P^{k_{\vec{W}}}[\vec{W}])$ is called the {\it minimum degree} of $\vec{W}$.

In fact, when we have a redundant state in the state space ${\cal X}$,
the kernel $\Ker P^{k_0}[\vec{W}]$ is not $\{0\}$.
For example, when the $d$-th element $x_d$ has the same behavior as the 
stochastic combination of $x_1, \ldots, x_{d-1}$ with the probabilities
$p_1, \ldots, p_{d-1}$,
the vector $(p_1, \ldots, p_{d-1},-1)$
belongs to the kernel $\Ker P^{k_0}[\vec{W}]$. 
For any integer $k$, 
we can naturally define the map $P^k[W]$ from $ {\cal V}_{{\cal X}}/\Ker P^{k}[\vec{W}]$ to ${\cal V}_{{\cal Y}^k}$.
That is, the distribution of $k$ outcomes of ${\cal Y}$
depends only on the element of 
the quotient space $ {\cal V}_{{\cal X}}/\Ker P^{k_{\vec{W}}}[\vec{W}]$.

\begin{lemma}\Label{L1T}
The minimum length $k_{\vec{W}}$ of $\vec{W}$ satisfies 
\begin{align}
k_{\vec{W}} & \le d \Label{8-21-1}\\
k_{\vec{W}} & \le 1+ \max_{y\in {\cal Y}}\rank W_y . \Label{8-21-2}
\end{align}
\end{lemma}
This lemma also shows the existence of $k_{\vec{W}}$.

\begin{proof}
\noindent{\it Step 1:\quad}
We will show the following fact;
If $\Ker P^k[\vec{W}] =\Ker P^{k+1}[\vec{W}] $,
$\Ker P^k[\vec{W}] =\Ker P^{k+l}[\vec{W}] $ for any $l \ge 0$.
We choose an arbitrary element $v \in \Ker P^{k+l}[\vec{W}]$.
For $(y_1, \ldots, y_{l-1}) \in \cY^{l-1}$, we have
$W_{y_{l-1}}\cdots W_{y_1} v \in \Ker P^{k+1}[\vec{W}]$, which implies that
$W_{y_{l-1}}\cdots W_{y_1} v \in \Ker P^{k}[\vec{W}]$.
So, $v\in \Ker P^{k+l-1}[\vec{W}]$.
Repeating this procedure, we obtain $v\in \Ker P^{k}[\vec{W}]$.

\noindent{\it Step 2:\quad}
Step 1 shows that 
$d = \dim \Ker P^0[\vec{W}] > \dim \Ker P^1[\vec{W}] > \ldots >  \dim \Ker P^{k_{\vec{W}}}[\vec{W}] \ge 0$, 
which implies \eqref{8-21-1}, i.e., 
$d \ge k_{\vec{W}}$.

\noindent{\it Step 3:\quad}
For an element $y \in \cY$, we choose the integer $k_y$ as the minimum integer $k_y$ satisfying
$\cap_{k=1}^{\infty} \Im W_y \cap \Ker P^{k}[\vec{W}] = \Im W_y \cap \Ker P^{k_y}[\vec{W}]$.
Now, we apply the above discussion to the case $\Im W_y $.
So, we have
$\dim \Im W_y = \dim \Im W_y \cap\Ker P^0[\vec{W}] > \dim \Im W_y \cap \Ker P^1[\vec{W}] > \ldots >  \dim \Im W_y \cap \Ker P^{k_y}[\vec{W}] \ge 0$,
which implies that
$\dim \Im W_y \ge k_y$.

\noindent{\it Step 4:\quad}
The relation $v \in \Ker P^{k+1}[\vec{W}]$ holds if and only if 
$W_y v \in \Ker P^{k}[\vec{W}]$ for any $y \in \cY$.
So, we have $\max_{y \in \cY} k_y +1 \ge k_{\vec{W}}$.
Combining Step 3, we have
\eqref{8-21-2}.
\hfill$\Box$\end{proof}

In the following, we regard a distribution $P$ on ${\cal X}$ as an element of $ {\cal V}_{{\cal X}}$.
Then, $[P]$ denotes an element of
the quotient space $ {\cal V}_{{\cal X}}/\Ker P^{k_{\vec{W}}}[\vec{W}]$
whose representative is $P$.
Given a distribution $P$ on $\cX$ and a positive integer $k$, 
we define the subspace ${\cal V}^k(P)$ of 
${\cal V}_{{\cal X}}/ \Ker P^{k_{\vec{W}}}[\vec{W}]$
spanned by $\{ [W_{y_{k'}}\cdots W_{y_1}P] | y_j \in \cY, k'\le k\}$.
We denote the minimum integer $k_0$ satisfying the following condition by $k_{(P,\vec{W})}$,
and call it the {\it minimum length} of $(P,\vec{W})$:
\begin{align}
\cup_{k =1}^{\infty} {\cal V}^k(P) = {\cal V}^{k_0}(P),
\end{align}
where the existence of the minimum is shown in Lemma \ref{L2T}.

\begin{lemma}
The space ${\cal V}^{k}(P_{\vec{W}})$
is spanned by $\big\{ [W_{y_k}\cdots W_{y_1}P_{\vec{W}}] \big| y_j \in \cY\big\}$.
\end{lemma}

\begin{proof}
It is enough to show that 
an element $[W_{y_{k-1}}\cdots W_{y_1}P_{\vec{W}}]$
is written as a linear combination of
$\{ [W_{y_k}\cdots W_{y_1}P_{\vec{W}}] | y_j \in \cY\}$.
Since $\sum_y W_{y}P_{\vec{W}} =P_{\vec{W}}$,  
we have $\sum_{y}[W_{y_{k-1}}\cdots W_{y_1} W_y P_{\vec{W}}]
=[W_{y_{k-1}}\cdots W_{y_1}P_{\vec{W}}]$.
\hfill$\Box$\end{proof}

\begin{lemma}\Label{L2T}
The minimum length $k_{(P,\vec{W})}$ of $\vec{W}$ exists and satisfies 
\begin{align}
k_{(P,\vec{W})} & \le d-\dim \Ker P^{k_{\vec{W}}}[\vec{W}]
 \Label{8-21-3}\\
k_{(P,\vec{W})} & \le 1+ \max_{y\in {\cal Y}} \dim  (\Im W_y/ \Ker P^{k_{\vec{W}}}[\vec{W}]) . \Label{8-21-4}
\end{align}
\end{lemma}

\begin{proof}
\noindent{\it Step 1:\quad}
We show the following fact;
If ${\cal V}^k(P) ={\cal V}^{k+1}(P) $,
${\cal V}^k(P) ={\cal V}^{k+l}(P) $ for any $l \ge 0$.
We choose an arbitrary element $v \in {\cal V}^{k+l}(P)$.
Choose an element $[W_{y_{k+l}}\cdots W_{y_1}P]$.
Due to the assumption, we have an element 
$[W_{y_{k+1}}\cdots W_{y_1}P]
=\sum_{\vec{y} \in \cup_{k'=1}^k \cY^{k'}} c_{\vec{y}}[W_{\vec{y}}P]$.
So, $$[W_{y_{k+l}}\cdots W_{y_1}P]
=\sum_{\vec{y} \in \cup_{k'=1}^k \cY^{k'}}
c_{\vec{y}}[
W_{y_{k+l}}\cdots W_{y_{k+2}} W_{\vec{y}}P]
\in {\cal V}^{k+l-1}(P).$$
Repeating this procedure, we have 
${\cal V}^k(P) ={\cal V}^{k+l}(P) $.

\noindent{\it Step 2:\quad}
Step 1 shows that 
\begin{align*}
& \dim {\cal V}_{\cX}/\Ker P^{k_{\vec{W}}}[\vec{W}]
\ge \dim {\cal V}^{k_{(P,\vec{W})}}(P) \\
> &
\dim {\cal V}^{k_{(P,\vec{W})}-1}(P) >\ldots >
\dim {\cal V}^{1}(P) >\dim {\cal V}^{0}(P)=1,
\end{align*}
which implies \eqref{8-21-3}, i.e., 
$\dim {\cal V}_{\cX}/\Ker P^{k_{\vec{W}}}[\vec{W}] \ge k_{(P,\vec{W})}$.

\noindent{\it Step 3:\quad}
For an element $y \in \cY$, we choose the integer $k_y'$ as the minimum integer $k_y'$ satisfying
$\cup_{k =1}^{\infty} [W_y]{\cal V}^k(P) = [W_y] {\cal V}^{k_y'}(P)$.
Replacing ${\cal V}^k(P)$
by $[W_y]{\cal V}^k(P)$ in the above discussion,
we have
$\dim \Im W_y/\Ker P^{k_{\vec{W}}}[\vec{W}] \ge k_{y}'$.

\noindent{\it Step 4:\quad}
We have $k_{(P,\vec{W})} \le 1+ \max_{y \in \cY} k_{y}'$.
Combining Step 3, we have
\eqref{8-21-4}.
\hfill$\Box$\end{proof}

\begin{lemma}\Label{L-1-25}
The relation
\begin{align}
W_y \Ker P^{k_{\vec{W}}}[\vec{W}]
\subset \Ker P^{k_{\vec{W}}}[\vec{W}]\Label{2-8-9}
\end{align}
holds for $y \in {\cal Y}$.
\end{lemma}
Due to Lemma \ref{L-1-25}, 
using $W_y$, we can define the linear map $[W_y]$ on  
the quotient space $ {\cal V}_{{\cal X}}/\Ker P^{k_{\vec{W}}}[\vec{W}]$.
We also define $d_{(P,\vec{W})}:= \dim {\cal V}^{k_{(P,\vec{W})}}(P)$.
Hence, the definition of ${\cal V}^{k_{(P,\vec{W})}}(P)$
and Lemma \ref{L-1-25} imply the following lemma.
\begin{lemma}
The following relation holds;
\begin{align}
[W_y] {\cal V}^{k_{(P,\vec{W})}}(P) \subset 
{\cal V}^{k_{(P,\vec{W})}}(P).\Label{2-8-10B}
\end{align}
\end{lemma}

\begin{proofof}{Lemma \ref{L-1-25}}
Assume that there exist elements $v \in \Ker P^{k_{\vec{W}}}[\vec{W}]$
and $y \in {\cal Y}$ such that
$ W_y v $ does not belong to $\Ker P^{k_{\vec{W}}}[\vec{W}]$.
Hence, 
$P^{k_{\vec{W}}}[\vec{W}] W_y v$ is not $0$.
Thus, $P^{k_{\vec{W}}+1}[\vec{W}] v$ is not $0$, which contradicts the assumption.
That is, we obtain \eqref{2-8-9}.
\end{proofof}

\begin{lemma}\Label{2-15-A}
The relations
$\Ker P^{k_{\vec{W}}}[\vec{W}]=\{0\}$ and ${\cal V}^{k_{P,\vec{W}}}(P)={\cal V}_{{\cal X}}$
hold almost everywhere with respect to $\vec{W}$ and $P$.
Also,
the relations
$\Ker P^{k_{\vec{W}}}[\vec{W}]=\{0\}$ and ${\cal V}^{k_{P,\vec{W}}}(P_{\vec{W}})={\cal V}_{{\cal X}}$
hold almost everywhere with respect to $\vec{W}$.
\end{lemma}

\begin{proof}
We show the first desired statement when $|\vec{W}|$ and $P$ are fixed, which is sufficient for both statements.
For an element $y\in {\cal Y}$, 
we can choose $W_y$ freely with the constraint that
$ |\vec{W}|- W_y $ is a positive matrix.
We choose $k$ to be $d-1$.
Since $W_y$ is freely chosen, 
the $d$ vectors
$|u_{\cal X}\rangle , W_y^T |u_{\cal X}\rangle  \ldots, (W_y^k)^T |u_{\cal X}\rangle $
are linearly independent almost everywhere with respect to the choice of $W_y$.
In this case, we have $\Ker P^{k}[\vec{W}]=\{0\}$, which implies the relation
$\Ker P^{k_{\vec{W}}}[\vec{W}]=\{0\}$.
Similarly,
the vectors $P, W_y P \ldots, W_y^k P$ are linear independent almost everywhere with respect to the choice of $W_y$.
In this case, 
since  ${\cal V}^{k}(P)$ contains $P, W_y P, W_y^2 P, \ldots, W_y^k P$
and $\Ker P^{k_{\vec{W}}}[\vec{W}]=\{0\}$,
we have ${\cal V}^{k}(P)={\cal V}_{{\cal X}}$, which implies that
${\cal V}^{k_{P,\vec{W}}}(P)={\cal V}_{{\cal X}}$.
\hfill$\Box$\end{proof}

\begin{remark}
The major part of this section is the reformulation of the result in \cite{IAK}.
Hence, some of the obtained statements are essentially given in \cite{IAK}.
For example, a statement similar to Lemma \ref{L2T} are given as \cite[Lemma 3]{IAK}.
Since \cite[Lemma 3]{IAK} shows that $k_{(P,\vec{W})} \le d$, but essentially shows \eqref{8-21-3}
while their formulation is different from ours.
However, they did not show \eqref{8-21-4}.
Since the paper \cite{IAK} did not consider $k_{P}$, Lemma \ref{L1T} is novel.
\end{remark}

\subsection{Equivalence relation}
In this subsection, we consider how to distinguish the pair of a ${\cal Y}$-indexed transition matrix $\vec{W}$ on ${\cal X}$
and a distribution $P$ on ${\cal X}$
from another pair of a ${\cal Y}$-indexed transition matrix $\vec{W}'$ on ${\cal X}'$
and a distribution $P'$ on ${\cal X}'$ from observed outcomes.
We say that the pair of $\vec{W}$ and $P$ is equivalent to the pair of $\vec{W}'$ and $P'$
when $P^{k}[\vec{W}] \cdot P=P^{k}[\vec{W}'] \cdot P'$ for any integer $k$.
Then, we obtain the following theorem as the refined version of Theorem \ref{sum-L-9-27}.

\begin{theorem}\Label{L-9-27}
The following conditions for 
two collections of non-negative matrices $\vec{W}$, $\vec{W}'$ 
and
two distributions $P$, $P'$ on $\cX$ 
are equivalent.
\begin{description}
\item[\bf (A1)]
There exists an invertible map $T$
from ${\cal V}^{k_{(P,\vec{W})}}(P)$
to ${\cal V}^{k_{(P',\vec{W}')}}(P')$ such that
the relation $T [W_y]  =[W_y'] T$ 
holds for $y\in {\cal Y}$ 
and the equation $T [P]=[P']$ holds.

\item[\bf (A1)']
{\bf (A1)} holds, and 
the relations $k_{\vec{W}}=k_{\vec{W}'}$,
$k_{(P,\vec{W})}=k_{(P',\vec{W}')}$, and $d_{(P,\vec{W})}=d_{(P',\vec{W}')}$
hold.

\item[\bf (A2)]
The pair of $\vec{W}$ and $P$ is equivalent to the pair of $\vec{W}'$ and $P'$.

\item[\bf (A3)]
The relations 
$k_{\vec{W}}=k_{\vec{W}'}$,
$k_{(P,\vec{W})}=k_{(P',\vec{W}')}$,
and
$P^{k_{\vec{W}}+k_{(P,\vec{W})}+1}[\vec{W}] \cdot P
=P^{k_{\vec{W}}+k_{(P',\vec{W}')}+1}[\vec{W}'] \cdot P'$ 
hold.

\item[\bf (A4)]
The relation 
$P^{k}[\vec{W}] \cdot P
=P^{k}[\vec{W}'] \cdot P'$ 
holds for $k=\max (k_{\vec{W}},k_{\vec{W}'})
+\max(k_{(P,\vec{W})},k_{(P',\vec{W}')})+1$.

\end{description}
Here, $[P]$ denotes an element of
the quotient space $ {\cal V}_{{\cal X}}/\Ker P^{k_{\vec{W}}}[\vec{W}]$
whose representative is $P$,
and $[W_y]$
denotes the linear map on 
the quotient space $ {\cal V}_{{\cal X}}/\Ker P^{k_{\vec{W}}}[\vec{W}]$
that is defined from $W_y$.
\end{theorem}

\begin{proof}
We notice that the relations {\bf (A1)} $\Rightarrow$ {\bf (A2)} $\Rightarrow$ 
{\bf (A4)}, and 
{\bf (A1)'} $\Rightarrow$ {\bf (A3)} $\Rightarrow$ {\bf (A4)} are trivial.
Assume {\bf (A1)}, 
since $T$ is invertible,
we can show 
the relations $k_{\vec{W}}=k_{\vec{W}'}$,
$k_{(P,\vec{W})}=k_{(P',\vec{W}')}$, and 
$d_{(P,\vec{W})}=d_{(P',\vec{W}')}$.
So, we have 
{\bf (A1)} $\Rightarrow$ {\bf (A1)'}.
Hence, it is enough to show that {\bf (A4)} $\Rightarrow$ {\bf (A1)}.


Assume {\bf (A4)}.
Now, we denote the $d$-dimensional linear space by $\bR^{d}$. 
Let $k_1:= \max(k_{(P,\vec{W})},k_{(P',\vec{W}')})$ and
$k_2:= \max(k_{\vec{W}},k_{\vec{W}'})$.
We choose elements $ \vec{y}_{1}, \ldots , \vec{y}_{d_{(P,\vec{W})}}  
\in \cup_{k \le k_1} \cY^{k}$ such that
$[W^{k_1}_{\vec{y}_{1}} \cdot P], \ldots, [W^{k_1}_{\vec{y}_{d_{(P,\vec{W})}} }\cdot P]$
are linear independent and span ${\cal V}^{k_{(P,\vec{W})}}(P)$.
The $d_{(P,\vec{W})} $ elements 
$ \vec{y}_{1}, \ldots , \vec{y}_{d_{(P,\vec{W})}}$
of $\cup_{k \le k_1} \cY^{k}$ give the invertible linear map $U_1$
from $\bR^{d_{(P,\vec{W})}}$ to 
${\cal V}^{k_{(P,\vec{W})}}(P)$ as
the map $(a_i)_{i=1}^{d_{(P,\vec{W})}}\mapsto
\sum_{i=1}^{d_{(P,\vec{W})}}a_i [W^{k_1}_{\vec{y}_{i}} \cdot P]$. 

We regard 
$P^{k_2+k_1}[\vec{W}] \cdot P$
as the joint distribution on $\cY^{k_2}$ and $\cY^{k_1}$.
Using the joint distribution, 
we define the transition matrix $M$ 
from the system $\cY^{k_1}$ to the system $\cY^{k_2}$.

When the matrix $M$ can be regarded as a linear map from 
${\cal V}^{k_1}(P)$ to ${\cal V}_{\cY^{k_2}}$,
it equals the map $P^{k_2}|_{{\cal V}^{k_1}(P)}$.
Since the map $P^{k_2}$ is the isomorphic linear map from
${\cal V}_{{\cal X}}$ to ${\cal V}_{\cY^{k_2}}$.
So, we find that the rank of the matrix $M$ is
the rank of ${\cal V}^{k_1}(P)$, which equals $ d_{(P,\vec{W})}$
and that ${\cal V}^{k_{(P,\vec{W})}}(P)$ is isomorphic to the image of the matrix $M$.
That is, the matrix $M|_{\tilde{\cal V}}$ is the isomorphic linear map from
$\tilde{\cal V}$ to the image of the matrix $M$.

In the same way, due to Condition {\bf (A4)}, we find that the rank of the matrix $M$ is $ d_{(P',\vec{W}')}$.
That is, $ d_{(P,\vec{W})}=d_{(P',\vec{W}')}$.
We define the invertible linear map $U_2$
from $\bR^{d_{(P,\vec{W})}}$ to 
${\cal V}^{k_{(P,\vec{W})}}(P)$ as
the map $(a_i)_{i=1}^{d_{(P,\vec{W})}}\mapsto
\sum_{i=1}^{d_{(P,\vec{W})}}a_i [{W'}^{k_1}_{\vec{y}_{i}} \cdot P']$. 

We choose a vector $(b_i)_{i=1}^{d_{(P,\vec{W})}} \in \bR^{d_{(P,\vec{W})}}$
such that
$\sum_{i=1}^{d_{(P,\vec{W})}} b_i [W^{k_1}_{\vec{y}_{i}}\cdot P]=[P]$.
Hence, we have $(b_i)_{i=1}^{d_{(P,\vec{W})}}=U_1^{-1}[P]$.
Due to Condition {\bf (A4)}, 
we also have
$\sum_{i=1}^{d_{(P',\vec{W}')}} b_i [{W'}^{k_1}_{\vec{y}_{i}}\cdot P']=[P']$, which implies that
$(b_i)_{i=1}^{d_{(P,\vec{W})}}=U_2^{-1}[P']$.
So, we have $U_1^{-1} [P]= U_2^{-1} [P']$, i.e., 
\begin{align}
U_2 U_1^{-1} [P]=  [P'].
\Label{8-21-9}
\end{align}

For $y \in \cY$, we define the $|\cY|^{k_1} \times |\cY|^{k_2}$ matrix $M_y$ by 
$(M_y)_{\vec{y}|\vec{y'}}:=
\sum_{x \in \cX} W_{\vec{y}}^{k_2}W_y W_{\vec{y}'}^{k_1} \cdot P (x)
/ \sum_{x' \in \cX} W_{\vec{y}'}^{k_1} \cdot P (x')$.
So, we have $M_y |_{\tilde{\cal V}}= M|_{\tilde{\cal V}} U_1^{-1} [W_y] U_1$.
In the same way, due to Condition (4), we find that 
$M_y |_{\tilde{\cal V}}= M|_{\tilde{\cal V}} U_2^{-1} [W_y'] {U_2}$.
Since the map $M|_{\tilde{\cal V}}$ is invertible, $U_1^{-1} [W_y] U_1= U_2^{-1} [W_y'] U_2$.
Defining $T:= U_2 U_1^{-1}$, we have $T [W_y]  =[W_y'] T$.
Combining \eqref{8-21-9}, we have Condition {\bf (A1)}.
\hfill$\Box$\end{proof}

Now, we assume that $\vec{W}$ and $\vec{W}'$ are irreducible ${\cal Y}$-indexed
transition matrix $\vec{W}$ on ${\cal X}$ and ${\cal X}'$.
So, we say that $\vec{W}$ is equivalent to $\vec{W}'$ 
when the pair of $\vec{W}$ and $P_{\vec{W}}$ is equivalent to the pair of $\vec{W}'$ and $P_{\vec{W}'}$.
Then, as a special case of Theorem \ref{L-9-27}, we have the following corollary.

\begin{corollary}
The following conditions for two ${\cal Y}$-indexed transition matrices $\vec{W}$ and $\vec{W}'$ are equivalent.
\begin{description}
\item[\bf (D1)]
There exists an invertible map $T$
from ${\cal V}^{k_{(P_{\vec{W}},\vec{W})}} (P_{\vec{W}})$
to ${\cal V}^{k_{(P_{\vec{W}'},\vec{W}')}} (P_{\vec{W}'})$
such that the relation 
$T [W_y]  =[W_y'] T$ for $y\in {\cal Y}$ 
holds as a linear map on the quotient space 
${\cal V}^{k_{(P_{\vec{W}},\vec{W})}} (P_{\vec{W}})$.

\item[\bf (F1')]
The conditions $k_{\vec{W}}=k_{\vec{W}'}$,
$k_{(P_{\vec{W}},\vec{W})}=k_{(P_{\vec{W}'},\vec{W}')}$,
and $d_{(P_{\vec{W}},\vec{W})}=d_{(P_{\vec{W}'},\vec{W}')}$
hold as well as the condition {\bf (D1)}.

\item[\bf (D2)]
$\vec{W}$ is equivalent to $\vec{W}'$.

\item[\bf (D3)]
The relations $P^{k_{\vec{W}}+k_{(P_{\vec{W}},\vec{W})}+1}[\vec{W}] \cdot P_{\vec{W}}
=P^{k_{\vec{W}'}+k_{(P_{\vec{W}'},\vec{W}')}+1}[\vec{W}'] \cdot P_{\vec{W}'}$, 
$k_{\vec{W}}=k_{\vec{W}'}$,
and
$k_{(P_{\vec{W}},\vec{W})}=k_{(P_{\vec{W}'},\vec{W}')}$
hold.

\item[\bf (D4)]
The relation 
$P^{k}[\vec{W}] \cdot P_{\vec{W}}
=P^{k}[\vec{W}'] \cdot P_{\vec{W}'}$ 
holds for $k=\max (k_{\vec{W}},k_{\vec{W}'})
+\max(k_{(P_{\vec{W}},\vec{W})},k_{(P_{\vec{W}'},\vec{W}')})+1$.
\end{description}
\end{corollary}

\begin{remark}
Theorem \ref{L-9-27} is similar to the main result of \cite{IAK}.
However, our treatment is different from that of \cite{IAK}.
Since the paper \cite{IAK} discusses only the equivalence condition in terms of 
the space ${\cal V}^k(P)$,
it treats only the integer $k_{P,\vec{W}}$ not the integer $k_{\vec{W}}$.
Therefore, it does not consider the condition using the integer $k_{\vec{W}}$.
That is, it shows only the equivalence between the conditions {\bf (A1)} and {\bf (A2)} in Theorem \ref{L-9-27}.
Hence, the discussion in \cite{IAK} cannot evaluate how large memory size $k$ is required to distinguish non-equivalent ${\cal Y}$-indexed transition matrices.
However, to employ the partial observation model to estimate the hidden Markov process,
we need to evaluate this number.
We discuss this number even with the first derivative of the observed joint distribution.
\end{remark}

\section{Exponential family of ${\cal Y}$-indexed transition matrices}\Label{s6}
\subsection{Definition of exponential family}\Label{s6-1}
To give a suitable parametrization,
we define an exponential family of ${\cal Y}$-indexed transition matrices.
Firstly, we fix an irreducible ${\cal Y}$-transition matrix 
$\vec{W}= (W_y(x|x'))_{y \in {\cal Y}}$ on ${\cal X}$.
Then, we denote the support of $\vec{W}$ by
$({\cal Y}\times {\cal X}^{2})_{\vec{W}}:=
\cup_{y\in\mathcal{Y}} \{y\}\times \mathcal{X}^2_{W_y}$.
Also, we denote the linear space of real-indexed functions 
$\{g(y,x,x')\}$ defined on $({\cal Y}\times {\cal X}^{2})_{\vec{W}}$
by ${\cal G}( ({\cal Y}\times {\cal X}^{2})_{\vec{W}})$.
Additionally, ${\cal N}(({\cal Y}\times {\cal X}^{2})_{\vec{W}}) $ expresses the subspace of functions 
$g(y,x,x')$ with form $ f(x)-f(x')+c$.

Now, we denote the vector $\{f(x)\}$ by $| f\rangle$.
Also, we denote $\{1\}_{x \in {\cal X}}$ and $\{1\}_{y \in {\cal Y}}$
by $u_{{\cal X}}$ and $u_{{\cal Y}}$. 
So, the element of ${\cal N}(({\cal Y}\times {\cal X}^{2})_{\vec{W}}) $ is written as
$ |u_{{\cal Y}}\rangle
(c|u_{{\cal X}}\rangle \langle u_{{\cal X}}|
+|f\rangle \langle u_{{\cal X}}|
-|u_{{\cal X}}\rangle \langle f|)$ by using a function $f$ on ${\cal X}$ and a constant $c$.

We define the linear map $\vec{W}_*$ on 
${\cal G}(({\cal Y}\times {\cal X}^{2})_{\vec{W}})$ as
\begin{align}
(\vec{W}_*g)(y,x,x'):=
g(y,x,x') W_{y}(x|x')
\end{align}
for $g\in {\cal G}(({\cal Y}\times {\cal X}^{2})_{\vec{W}})$.
Then, we define the subspaces of ${\cal G}( ({\cal Y}\times {\cal X}^{2})_{\vec{W}})$ as
\begin{align}
{\cal L}_{1,\vec{W}} &:=
\Big\{ (B_y)_{y \in {\cal Y}} \in {\cal G}(({\cal Y}\times {\cal X}^{2})_{\vec{W}}) \Big|
\sum_y B_y^T | u_{{\cal X}}\rangle =0 \Big\} \nonumber \\
{\cal G}_1( ({\cal Y}\times {\cal X}^{2})_{\vec{W}})
&:=
\vec{W}_*^{-1} ({\cal L}_{1,\vec{W}}).
\end{align}

Then, as shown in the following lemma, 
${\cal G}_1( ({\cal Y}\times {\cal X}^{2})_{\vec{W}})$ can be identified with
the quotient space ${\cal G}(({\cal Y}\times {\cal X}^{2})_{\vec{W}})/{\cal N}(({\cal Y}\times {\cal X}^{2})_{\vec{W}})$.
\begin{lemma}
For any element $[g']$ of
${\cal G}(({\cal Y}\times {\cal X}^{2})_{\vec{W}})/{\cal N}(({\cal Y}\times {\cal X}^{2})_{\vec{W}})$,
there uniquely exists an element $g $ of ${\cal G}_1( ({\cal Y}\times {\cal X}^{2})_{\vec{W}})$
such that $[g']=[g]$, i.e., $g$ is a representative of $[g']$.
Therefore, we can regard the space ${\cal G}_1( ({\cal Y}\times {\cal X}^{2})_{\vec{W}})$
as the quotient space
${\cal G}(({\cal Y}\times {\cal X}^{2})_{\vec{W}})/{\cal N}(({\cal Y}\times {\cal X}^{2})_{\vec{W}})$.
\end{lemma}

\begin{proof}
\noindent{\it Step 1:\quad}
We will show that
\begin{align}
{\cal V}_{{\cal X}}= {\cal N}_{\vec{W}} :=\{ (\vec{W}_*  g)^T  |u_{{\cal X}}\rangle
| g \in {\cal N}( ({\cal Y}\times{\cal X}^{2})_{\vec{W}}) \}.\Label{9-29-1}
\end{align}
Here, we regard an element of $ {\cal N}( ({\cal Y}\times{\cal X}^{2})_{\vec{W}})$
as a matrix on ${\cal V}_{\cal X}$.

For this purpose, we will show that 
the function $|f\rangle- \langle f| P_W\rangle |u_{{\cal X}}\rangle$ 
belongs to the RHS of \eqref{9-29-1} for any function $f$.
Since 
\begin{align}
&( ( {\vec{W}}_* (- |f\rangle \langle u_{{\cal X}}| +|u_{{\cal X}}\rangle \langle f|  ) u_{{\cal X}})_{x'}\nonumber \\
=&\sum_{x}|\vec{W}|(x|x')(-f(x)+f(x')) \nonumber \\
=&f(x')-\sum_{x}f(x)|\vec{W}|(x|x')
=f(x')-(|\vec{W}|^T |f\rangle)_{x'},
\end{align}
$ |f\rangle- |\vec{W}|^T |f\rangle$ belongs to the set ${\cal N}_{\vec{W}}$.
So, 
$ |f\rangle- |\vec{W}|^T| f\rangle +|\vec{W}|^T| f\rangle- |\vec{W}|^T |\vec{W}|^T |f\rangle 
=| f\rangle-|\vec{W}|^T |\vec{W}|^T |f\rangle$ belongs to the set ${\cal N}_{\vec{W}}$.
Repeating this procedure, we see that
$|f\rangle - (|\vec{W}|^T)^n |f\rangle$ belongs to the set ${\cal N}_{\vec{W}}$.
Since $\lim_{n \to \infty}\frac{1}{n}\sum_{i=1}^n \sum_{x}f(x)|\vec{W}|^i(x|x')= \sum_{x}f(x)P_{ |\vec{W}|}(x) 
=\langle f| P_{ |\vec{W}|} \rangle $ for any $x' \in {\cal X}$,
we have $\lim_{n \to \infty }\frac{1}{n}\sum_{i=1}^n|f\rangle - (|\vec{W}|^T)^i |f\rangle
=|f\rangle- \langle f| P_{ |\vec{W}|}\rangle |u_{{\cal X}}\rangle$, i.e., 
$|f\rangle- \langle f| P_{ |\vec{W}|}\rangle |u_{{\cal X}}\rangle$ belongs to the set ${\cal N}_{\vec{W}}$.
Since 
\begin{align}
( {\vec{W}}_*  |u_{{\cal X}}\rangle \langle  u_{{\cal X}}| )| u_{{\cal X}}\rangle_{x'}
 = \sum_{x}|\vec{W}|(x|x')=1= |u_{{\cal X}}\rangle_{x'},
\end{align}
$ |u_{{\cal X}}\rangle $ belongs to the set ${\cal N}_{\vec{W}}$.
Thus, any function $f$ belongs to the set ${\cal N}_{\vec{W}}$, which implies 
\eqref{9-29-1}.

\noindent{\it Step 2:\quad}
Given $g' \in {\cal G}( ({\cal Y}\times{\cal X}^{2})_{\vec{W}})$, 
we can choose an element $g'' \in {\cal N}( ({\cal Y}\times{\cal X}^{2})_{\vec{W}})$
such that $( \sum_{y \in {\cal Y}}(\vec{W}_* g')_y )^T| u_{{\cal X}}\rangle  =( \vec{W}_* g'')^T| u_{{\cal X}}\rangle $.
That is, when we choose $g:= g'-\frac{1}{d_Y}g''$,
Then, 
\begin{align}
&\sum_{y \in {\cal Y}} (\vec{W}_* g)_y^T | u_{{\cal X}}\rangle
=
\sum_{y \in {\cal Y}} (\vec{W}_* g')_y^T | u_{{\cal X}}\rangle
-
\frac{1}{d_Y} \sum_{y \in {\cal Y}} \vec{W}_* g''^T | u_{{\cal X}}\rangle
\nonumber \\
=&
\sum_{y \in {\cal Y}} (\vec{W}_* g')_y^T | u_{{\cal X}}\rangle
-
 \vec{W}_* g''^T | u_{{\cal X}}\rangle
=0.
\end{align}
Hence, $g$ belongs to ${\cal G}_{1,W}( ({\cal Y}\times{\cal X}^{2})_{\vec{W}})$.
\hfill$\Box$\end{proof}

When functions $g_1, \ldots, g_l \in {\cal G}(({\cal Y}\times {\cal X}^{2})_{\vec{W}})$
are linearly independent as elements of ${\cal G}(({\cal Y}\times {\cal X}^{2})_{\vec{W}})/{\cal N}(({\cal Y}\times {\cal X}^{2})_{\vec{W}})$,
for $\vec{\theta}:=(\theta^1, \ldots, \theta^l) \in \bR^l$,
we define the matrix 
$\overline{W}_{\vec{\theta}}(x|x')
:=\sum_{y} e^{\sum_{j=1}^l \theta^j g_j(y, x,x')} W_y(x|x')$,
and denote the Perron-Frobenius eigenvalue by $\lambda_{\vec{\theta}}$ 
Also, we denote the Perron-Frobenius eigenvector of 
the transpose $\overline{W}_{\vec{\theta}}^T$
by $\overline{P}^3_{\vec{\theta}}$\footnote{For the 
Perron-Frobenius eigenvalue and Perron-Frobenius eigenvector, see the references \cite[Theorem 3.1.]{DZ}\cite{Sen}.}.

Then, we define the ${\cal Y}$-indexed transition matrix
$\vec{W}_{\vec{\theta}}=(W_{\vec{\theta},y})_{y \in {\cal Y}}$ on ${\cal X}$ as
$W_{\vec{\theta},y}(x, x'):=
\lambda_{\vec{\theta}}^{-1} \overline{P}^3_{\vec{\theta}}(x)
e^{\sum_{j=1}^l \theta^j g_j(y,x,x')}
W_y(x,x') \overline{P}^3_{\vec{\theta}}(x')^{-1}$,
and 
$\{\vec{W}_{\vec{\theta}}\}_{\vec{\theta}}$
is called an exponential family of 
${\cal Y}$-indexed transition matrices on ${\cal X}$
generated by the generators $\{g_i\}_{i=1}^l$.
Since a ${\cal Y}$-indexed transition matrix $\vec{W}$ can be regarded as a transition matrix on 
${\cal X}\times {\cal Y}$ as \eqref{LGR},
$\{ \vec{W}_{\vec{\theta}| {\cal X}\times {\cal Y}} \}_{\vec{\theta}}$
forms an exponential family of transition matrices on 
${\cal X}\times {\cal Y}$.

Here, we check that the exponential family defined here coincides with the exponential family on 
$\tilde{{\cal X}}={\cal X} \times {\cal Y}$.
We regard the functions $g_i$ as functions on $\tilde{{\cal X}}^2$ as $g_i(x,y,x'y'):=g_i(x,y,x') $.
Then, we can define the non-negative matrix $\overline{\vec{W}}_{\tilde{{\cal X}},\vec{\theta}}$ 
on $\tilde{{\cal X}}$.
When $\overline{P}^3_{\vec{\theta}}$ is regarded as a vector on $\tilde{{\cal X}}$
in the sense $\overline{P}^3_{\vec{\theta}}(x,y)=\overline{P}^3_{\vec{\theta}}(x)$,
$\overline{P}^3_{\vec{\theta}}$ is also 
the Perron-Frobenius eigenvector of
the transpose of the non-negative matrix $\overline{\vec{W}}_{\tilde{{\cal X}},\vec{\theta}}$.
Then, we find that $\lambda_{\vec{\theta}}$ is also the Perron-Frobenius eigenvector of
the non-negative matrix $\overline{\vec{W}}_{\tilde{{\cal X}},\vec{\theta}}$.
Therefore, the exponential family 
$\overline{\vec{W}}_{\tilde{{\cal X}},\vec{\theta}}$
satisfies that
$\overline{\vec{W}}_{\tilde{{\cal X}},\vec{\theta}}(x,y|x',y')= W_{\vec{\theta},y}(x, x')$
because 
$\vec{W}_{\tilde{{\cal X}},\vec{\theta}} (x,y|x',y')
=\lambda_{\vec{\theta}}^{-1} \overline{P}^3_{\vec{\theta}}(x)
e^{\sum_{j=1}^l \theta^j g_j(y,x,x')} W_y(x|x') \overline{P}^3_{\vec{\theta}}(x')^{-1}$.
The eigenvector of $\vec{W}_{\tilde{{\cal X}},\vec{\theta}} $ is given 
by the relation \eqref{Eq3-7-1}.

In this sense, we call $\phi(\vec{\theta}):=\log \lambda_{\vec{\theta}}$ the potential function. 
Then, we define the divergence between two ${\cal Y}$-indexed transition matrices as
\begin{align}
D\big(\vec{W}_{\vec{\theta}| {\cal X}\times {\cal Y}} 
\big\| \vec{W}_{\vec{\theta}'| {\cal X}\times {\cal Y}} \big)
=
\sum_{j=1}^d
(\theta^j-{\theta'}^j)\frac{\partial \phi}{\partial \theta^j}(\vec{\theta})- \phi(\vec{\theta})+ \phi(\vec{\theta}') \Label{1-1X},
\end{align}
which is a special case of divergence between two transition matrices on $\tilde{\cal X}$ defined in \cite{NK,HN}. 
So, we call an element of the space ${\cal G}_1( ({\cal Y}\times {\cal X}^{2})_{\vec{W}})$
and the quotient space ${\cal G}(({\cal Y}\times {\cal X}^{2})_{\vec{W}})/{\cal N}(({\cal Y}\times {\cal X}^{2})_{\vec{W}})$ the $e$-representation, and call an element of ${\cal L}_{1,\vec{W}}$ the $m$-representation.

\begin{example}\Label{VDE}
We consider the full model, i.e., the set of ${\cal Y}$-indexed transition matrices $\vec{W}$ 
satisfying that all the components of $W_y$ are non zero, i.e.,
$({\cal Y}\times {\cal X}^{2})_{\vec{W}}:={\cal Y}\times {\cal X}^{2}$.
Hence, we choose a ${\cal Y}$-indexed transition matrix $\vec{W}$ satisfying this condition. 
Since the dimension of ${\cal N}(({\cal Y}\times {\cal X}^{2})_{\vec{W}})$ is $d $ and
the dimension of ${\cal G}(({\cal Y}\times {\cal X}^{2})_{\vec{W}})$ is $ d^2\cdot d_Y$,
the dimension of the quotient space given in \eqref{eq1-26} is 
$l:=d^2\cdot d_Y- d
= d\cdot (d \cdot d_Y-1) $.

Unfortunately, it is not easy to choose elements
$g_{1}, \ldots, g_{l}$ to be elements of ${\cal G}_1(({\cal Y}\times {\cal X}^{2})_{\vec{W}})$ as generators of 
an exponential family of ${\cal Y}$-indexed transition matrices.
Hence, in the following, we choose $l$ functions
$g_{1}, \ldots, g_{l}$ to be elements of 
${\cal G}(({\cal Y}\times {\cal X}^{2})_{\vec{W}})$.
In this case, we can easily find the generators as follows.
Here, we do not necessarily choose the generators from 
${\cal G}_1(({\cal Y}\times {\cal X}^{2})_{\vec{W}})$.
That is, it is sufficient to choose them as elements of 
${\cal G}(({\cal Y}\times {\cal X}^{2})_{\vec{W}})$.
We define $g_{j+(i-1) d_Y +(i'-1) (d d_Y-1) }$ for $i,i'=1,\ldots,d $ and $j=1, \ldots, d_Y$ as follows.
However, when $i=d$, the index $j$ runs from $1$ to $d_Y-1$. 
\begin{align}
g_{j+(i'-1) d_Y +(i-1) (d d_Y-1) }(y,x,x') &:= 
\delta_{y,j}\delta_{x,i}\delta_{x',i'}
\Label{LR1} 
\end{align}
Then, we obtain an exponential family of ${\cal Y}$-indexed transition matrices
generated by $g_{1}, \ldots, g_{l}$ at $\vec{W}$.
Among $l$ generators, 
we can directly observe $l'$ generators at most.
Since the dimension of the quotient space generated by ${\cal G}({\cal Y})_{\vec{W}}$ is 
$d_Y-1 $.
\end{example}

\subsection{Local equivalence}
Although we give an example of 
an exponential family of ${\cal Y}$-indexed transition matrices,
we cannot necessarily distinguish element of this exponential family
from observed data in ${\cal Y}$ due to the equivalence problem.
To discuss this equivalence relation among generators,
we introduce other subspaces 
as follows.
For this am, we denote the set of linear maps on ${\cal V}_1$ by ${\cal M}({\cal V}_1)$,
and we identify an element of ${\cal G}(({\cal Y}\times {\cal X}^{2})_{\vec{W}})$
with a vector taking values in ${\cal M}({\cal V}_{{\cal X}})$.
Then, for a distribution $P $ on $\cX$,
we define the subspaces 
${\cal L}_{2,\vec{W}} $, 
${\cal L}_{P,\vec{W}} $, and  
${\cal L}_{2,P,\vec{W}} $ 
of the linear space composed of vectors taking values in the matrix space ${\cal M}({\cal V}_{{\cal X}})$;
\begin{align*}
{\cal L}_{2,\vec{W}} :
=&
\big\{ (\alpha_y(A) )_{y\in {\cal Y}} \big| 
A \in {\cal M}({\cal V}_{{\cal X}}),~c\in \mathbb{R},~
A^T |u_{{\cal X}}\rangle=c |u_{{\cal X}}\rangle \big\} \\
=&
\big\{ (\alpha_y(A) )_{y\in {\cal Y}} \big| 
A \in {\cal M}({\cal V}_{{\cal X}}),~
A^T |u_{{\cal X}}\rangle=0 \big\} \\
{\cal L}_{P,\vec{W}} :=&
\Bigg\{ (B_{y})_{y\in {\cal Y}} 
\in 
{\cal L}_{1,\vec{W}}
\Bigg| 
\begin{array}{l}
B_y ({\cal V}^{k_{(P,\vec{W})}}(P) +\Ker P^{k_{\vec{W}}}[\vec{W}] )
\subset \Ker P^{k_{\vec{W}}}[\vec{W}]\\
(\sum_y B_y)^T |u_{{\cal X}}\rangle=0 
\end{array}
\Bigg\} \\
{\cal L}_{2,P,\vec{W}} :=&
\big\{ (\alpha_y(A))_{y\in {\cal Y}} \big| 
A \in {\cal M}({\cal V}_{{\cal X}})
\hbox{ such that } A| P\rangle=0,~
A^T |u_{{\cal X}}\rangle=0
\big\} ,
\end{align*}
where $\alpha_y(A):=[W_y,A]$.
So, we define 
${\cal N}_2(({\cal Y}\times {\cal X}^{2})_{\vec{W}}) :=
\vec{W}_*^{-1}({\cal L}_{2,\vec{W}} )$,
${\cal N}_P(({\cal Y}\times {\cal X}^{2})_{\vec{W}}) :=
\vec{W}_*^{-1}({\cal L}_{P,\vec{W}} )$,
${\cal N}_{2,P}(({\cal Y}\times {\cal X}^{2})_{\vec{W}}) :=
\vec{W}_*^{-1}({\cal L}_{2,P,\vec{W}} )$.
Since the following theorems show that 
the infinitesimal change of an element of these subspaces cannot be distinguished 
from the observed data,
these subspaces are called {\it indistinguishable subspaces}. 
Then, we obtain the following theorem as the refined version of Theorem \ref{sum-L27-2}.

\begin{theorem}\Label{L27-2}
Given an irreducible ${\cal Y}$-indexed transition matrix $\vec{W}$, and a distribution $P$ on ${\cal X}$,
the following conditions are equivalent for 
functions $g_1, \ldots, g_l \in {\cal G}_1(({\cal Y}\times {\cal X}^{2})_{\vec{W}})$
and a vector $\vec{a} \in \mathbb{R}^l$,
where $\{\vec{W}_{\vec{\theta}}\}_{\vec{\theta}}$ is the exponential family of 
${\cal Y}$-indexed transition matrices on ${\cal X}$
generated by the generators $\{g_i\}_{i=1}^l$.

\begin{description}
\item[\bf (B1)]
The function $\sum_{j=1}^l a^j g_j \in {\cal G}_1(({\cal Y}\times {\cal X}^{2})_{\vec{W}})$
belongs to 
$
{\cal N}_P(({\cal Y}\times {\cal X}^{2})_{\vec{W}})
+{\cal N}_{2,P}(({\cal Y}\times {\cal X}^{2})_{\vec{W}})$.
\item[\bf (B2)]
The relation $\sum_{j=1}^l a^j \frac{\partial }{\partial \theta^j}
P^{k}[\vec{W}_{\vec{\theta}}]\cdot P \Big|_{\vec{\theta}=0}= 0$ holds for any positive integer $k$.

\item[\bf (B3)]
The relation $\sum_{j=1}^l a^j \frac{\partial }{\partial \theta^j}
P^{k_{\vec{W}}+k_{(P,\vec{W})}+1}[\vec{W}_{\vec{\theta}}]\cdot P \Big|_{\vec{\theta}=0}= 0$ holds.

\item[\bf (B4)]
The relation $\sum_{j=1}^l a^j \frac{\partial }{\partial \theta^j}
P^{k}[\vec{W}_{\vec{\theta}}]\cdot P \Big|_{\vec{\theta}=0}= 0$ holds 
with a certain integer $k \ge k_{(W,V)}+k_{(P_{(W,V)},(W,V))}+1$.
\end{description}
\end{theorem}
Theorem \ref{L27-2} is shown in Appendix \ref{A1}.
Since we have the relation {\bf (B2)} $\Rightarrow$ {\bf (B4)'} $\Rightarrow$ {\bf (B4)},
Theorem \ref{L27-2} implies Theorem \ref{sum-L27-2}.
When the vector $\vec{a}\in \mathbb{R}^l$ is identified with $\sum_{j=1}^l a^j g_j$, 
the local equivalence class at $\theta$ with the initial distribution $P$ is given as 
the space 
\begin{align*}
&{\cal G}_1(({\cal Y}\times {\cal X}^{2})_{\vec{W}})/
({\cal N}_P(({\cal Y}\times {\cal X}^{2})_{\vec{W}})
+{\cal N}_{2,P}(({\cal Y}\times {\cal X}^{2})_{\vec{W}})) \\
=&
{\cal G}(({\cal Y}\times {\cal X}^{2})_{\vec{W}})/
({\cal N}(({\cal Y}\times {\cal X}^{2})_{\vec{W}})
+{\cal N}_P(({\cal Y}\times {\cal X}^{2})_{\vec{W}})
+{\cal N}_{2,P}(({\cal Y}\times {\cal X}^{2})_{\vec{W}})).
\end{align*}

The above discussion addresses the equivalence when the initial distribution is fixed to be $P$.
However, in the asymptotic case, the distribution converges to the stationary distribution $P_{\vec{W}_{\vec{\theta}}}$.
To address this case, we have the following theorem
as the refined version of Theorem \ref{sum-27-3}.

\begin{theorem}\Label{27-3}
Given an irreducible ${\cal Y}$-indexed transition matrix $\vec{W}$, 
the following conditions are equivalent for 
functions $g_1, \ldots, g_l \in {\cal G}_{1}(({\cal Y}\times {\cal X}^{2})_{\vec{W}})$ 
and a vector $\vec{a} \in \mathbb{R}^l$
under the same condition as Theorem \ref{L27-2}.
\begin{description}
\item[\bf (C1)]
The function $\sum_{j=1}^l a^j g_j \in {\cal G}_{1}(({\cal Y}\times {\cal X}^{2})_{\vec{W}})$
belongs to $
{\cal N}_2(({\cal Y}\times {\cal X}^{2})_{\vec{W}})
+{\cal N}_{P_{\vec{W}}}(({\cal Y}\times {\cal X}^{2})_{\vec{W}})$.

\item[\bf (C2)]
The relation $\sum_{j=1}^l a^j \frac{\partial }{\partial \theta^j}
P^{k}[\vec{W}_{\vec{\theta}}]
\cdot P_{\vec{W}_{\vec{\theta}}} \Big|_{\vec{\theta}=0}
= 0$ holds for any positive integer $k$.

\item[\bf (C3)]
The relation $\sum_{j=1}^l a^j \frac{\partial }{\partial \theta^j}
P^{k_{\vec{W}}+k_{(P_{\vec{W}},\vec{W})}+1}
[\vec{W}_{\vec{\theta}}]
\cdot P_{\vec{W}_{\vec{\theta}}} \Big|_{\vec{\theta}=0} = 0$ holds.

\item[\bf (C4)]
The relation $\sum_{j=1}^l a^j \frac{\partial }{\partial \theta^j}
P^{k}[\vec{W}_{\vec{\theta}}]
\cdot P_{\vec{W}_{\vec{\theta}}} \Big|_{\vec{\theta}=0} = 0$ holds
with a certain integer $k \ge k_{(W,V)}+k_{(P_{(W,V)},(W,V))}+1$.
\end{description}
\end{theorem}

Theorem \ref{27-3} is shown in Appendix \ref{A2}.
Since we have the relation {\bf (C2)} $\Rightarrow$ {\bf (C4)'} $\Rightarrow$ {\bf (C4)},
Theorem \ref{27-3} implies Theorem \ref{sum-27-3}.
Due to this theorem, under the above identification, 
the local and asymptotic equivalence class at $\theta$ is given as 
the space 
\begin{align}
&{\cal G}_1(({\cal Y}\times {\cal X}^{2})_{\vec{W}})/
({\cal N}_{P_{\vec{W}}}(({\cal Y}\times {\cal X}^{2})_{\vec{W}})+{\cal N}_2(({\cal Y}\times {\cal X}^{2})_{\vec{W}}))
\nonumber \\
=&
{\cal G}(({\cal Y}\times {\cal X}^{2})_{\vec{W}})/
({\cal N}(({\cal Y}\times {\cal X}^{2})_{\vec{W}})
+{\cal N}_{P_{\vec{W}}}(({\cal Y}\times {\cal X}^{2})_{\vec{W}})
+{\cal N}_2(({\cal Y}\times {\cal X}^{2})_{\vec{W}}))
\nonumber \\
=&
{\cal G}(({\cal Y}\times {\cal X}^{2})_{\vec{W}})/
({\cal N}_3(({\cal Y}\times {\cal X}^{2})_{\vec{W}})
+{\cal N}_{P_{\vec{W}}}(({\cal Y}\times {\cal X}^{2})_{\vec{W}})),
\Label{eq1-26}
\end{align}
where 
\begin{align}
{\cal N}_3(({\cal Y}\times {\cal X}^{2})_{\vec{W}}):=
{\cal N}(({\cal Y}\times {\cal X}^{2})_{\vec{W}})+
{\cal N}_2(({\cal Y}\times {\cal X}^{2})_{\vec{W}}).
\end{align}

When the generators of our exponential family
are not linearly independent in the sense of 
${\cal G}(({\cal Y}\times {\cal X}^{2})_{\vec{W}})/
({\cal N}(({\cal Y}\times {\cal X}^{2})_{\vec{W}})
+{\cal N}_{P_{\vec{W}}}(({\cal Y}\times {\cal X}^{2})_{\vec{W}})
+{\cal N}_2(({\cal Y}\times {\cal X}^{2})_{\vec{W}}))
$, the parametrization around $\vec{W}$
does not express distinguishable information.
That is, the parametrization is considered to be redundant.

\begin{lemma}\Label{L2-13}
The space ${\cal N}_3(({\cal Y}\times {\cal X}^{2})_{\vec{W}})$ 
is characterized as
\begin{align}
&\vec{W}_* ({\cal N}_3(({\cal Y}\times {\cal X}^{2})_{\vec{W}})) \nonumber \\
=&
\big\{ (\alpha_y(A) + c W_y)_{y\in {\cal Y}}  \big| 
c\in \mathbb{R},~
A \in {\cal M}({\cal V}_{{\cal X}}),~
\langle u_{{\cal X}}| A| u_{{\cal X}}\rangle=0 \big\} .
\Label{eq2-13}
\end{align}
\end{lemma}

\begin{proofof}{Lemma \ref{L2-13}}
For a function $f$ on ${\cal X}$, we define the diagonal matrix $D_f$ on $V_{{\cal X}}$
whose diagonal element is $f(x)$.
The, an element 
$ |u_{{\cal Y}}\rangle \otimes
(c|u_{{\cal X}}\rangle \langle u_{{\cal X}}|
+|f\rangle \langle u_{{\cal X}}|
-|u_{{\cal X}}\rangle \langle f|)
\in {\cal N}(({\cal Y}\times {\cal X}^{2})_{\vec{W}}) $ satisfies that
\begin{align}
& \vec{W}_* ( |u_{{\cal Y}}\rangle\otimes
(c|u_{{\cal X}}\rangle \langle u_{{\cal X}}|
+|f\rangle \langle u_{{\cal X}}|
-|u_{{\cal X}}\rangle \langle f|) )
=
([W_y,-D_f]+c W_y )_{y \in {\cal Y}}.
\end{align}
Therefore, any element of LHS of \eqref{eq2-13}
can be written as
$([W_y,A -D_f ]+c W_y 
)_{y \in {\cal Y}}$ by using 
a function $f$ on ${\cal X}$, $c \in \mathbb{R}$, and
a matrix $A\in {\cal M}({\cal X})$ satisfying $ A^T|u_{\cal X}\rangle=0$.
Since the matrix 
 $A':= A -D_f-  
\frac{1}{d^2} \langle u_{{\cal X}}| (A -D_f) |u_{{\cal X}} \rangle
 |u_{{\cal X}} \rangle \langle u_{{\cal X}}| $
satisfies 
$ \langle u_{{\cal X}}| A' |u_{{\cal X}} \rangle=0$
and $[W_y,A -D_f ]=[W_y,A']$.
Then, we obtain the desired statement.
\end{proofof}

\section{Construction of linearly independent generators}\Label{s6-3}
In this section, we construct generators for the full model
such that they are linear independent in the sense of the quotient space
$ {\cal G}(({\cal Y}\times {\cal X}^{2})_{\vec{W}})/
({\cal N}(({\cal Y}\times {\cal X}^{2})_{\vec{W}})
+{\cal N}_{P_{\vec{W}}}(({\cal Y}\times {\cal X}^{2})_{\vec{W}})
+{\cal N}_2(({\cal Y}\times {\cal X}^{2})_{\vec{W}}))$.

For this aim, 
we consider the following conditions for $\vec{W}$.
\begin{description}
\item[\bf (E1)]
$\Ker P^{k_{\vec{W}}}[\vec{W}]=\{0\}$ and ${\cal V}^{k_{P,\vec{W}}}(P)={\cal V}_{{\cal X}}$.
\item[\bf (E2)]
All of the components of $W_y$ are non zero, i.e.,
$({\cal Y}\times {\cal X}^{2})_{\vec{W}}:={\cal Y}\times {\cal X}^{2}$.
\end{description}
Lemma \ref{2-15-A} guarantees that Condition {\bf (E1)} holds almost everywhere.

Under these conditions, 
using the notations $d:=d$ and $d_Y:=|{\cal Y}|$,
we consider the full parameter model, and choose $k$ to be greater than or equal to $k_{\vec{W}}$.
So, {\bf (E1)} implies $ {\cal L}_{P_{\vec{W}},\vec{W}} =\{0\}$, i.e., 
${\cal N}_{P_{\vec{W}}}(({\cal Y}\times {\cal X}^{2})_{\vec{W}})=\{0\}$.
{\bf (E2)} guarantees that the dimension of ${\cal L}_{2,\vec{W}}$ is $ d^2-d$.
So, {\bf (E2)} guarantees that 
the dimension of ${\cal N}_{2}(({\cal Y}\times {\cal X}^{2})_{\vec{W}})$ is $d^2-d$.
Since the dimension of ${\cal N}(({\cal Y}\times {\cal X}^{2})_{\vec{W}})$ is $d $ and
the dimension of ${\cal G}(({\cal Y}\times {\cal X}^{2})_{\vec{W}})$ is $ d^2\cdot d_Y$,
the dimension of the quotient space given in \eqref{eq1-26} is 
$l:=d^2\cdot d_Y- (d^2-d)- d
= d^2\cdot (d_Y-1) $.
Since Condition {\bf (E2)} holds almost everywhere as well,
when we fix ${\cal X}$ and ${\cal Y}$,
these discussions show that the dimension of the tangent space
${\cal G}(({\cal Y}\times {\cal X}^{2})_{\vec{W}})/
({\cal N}_3(({\cal Y}\times {\cal X}^{2})_{\vec{W}})+{\cal N}_{P_{\vec{W}}}(({\cal Y}\times {\cal X}^{2})_{\vec{W}}))$ 
is $d^2\cdot (d_Y-1)$ almost everywhere.
However, in several points, the dimension is strictly smaller than this value.
We call such points {\it singular points}.

Next, at the neighborhood of a non-singular point, we give generators. 
For this aim, in addition to Conditions {\bf (E1)} and {\bf (E2)}, we assume the following condition.
\begin{description}
\item[\bf (E3)]
There exist two elements $y_0,y_1\in {\cal Y}$ such that
(1) the map $\alpha_{y_0}$ is injective on the set $\{ A \in {\cal M}({\cal V}_{\cal X}) |  A^T |u_{{\cal X}}\rangle=0 \}$
and (2) the map $A\mapsto (\alpha_{y_0}(A),\alpha_{y_1}(A))$ is injective on the set $\{ A \in {\cal M}({\cal V}_{\cal X}) | \langle u_{{\cal X}}| A|u_{{\cal X}}\rangle=0 \}$.
\end{description}

For Condition {\bf (E3)}, we have the following lemma.
\begin{lemma}\Label{L1-26-2}
Assume that $W_{y_0}^T$ and $W_{y_1}^T$ have $d$ distinct eigenvalues and their eigenvectors
$f_1, \ldots, f_{d}$ and $f_1', \ldots, f_{d}'$, respectively.
Also, assume that $u_{{\cal X}}=\sum_{j=1}^d a^j f_j$ and $a^j \neq 0$ for any $j$. 
Then, the condition (1) of Condition {\bf (E3)} holds.
Additionally, we assume that the eigenvectors $f_1', \ldots, f_{d}'$ 
are distinct from the eigenvectors $f_1, \ldots, f_{d}$.
Then, the condition (2) of Condition {\bf (E3)} holds.
\end{lemma}

Unfortunately, it is not easy to choose elements
$g_{1}, \ldots, g_{l}$ to be elements of ${\cal G}_1(({\cal Y}\times {\cal X}^{2})_{\vec{W}})$ as generators of 
an exponential family of ${\cal Y}$-indexed transition matrices.
Hence, in the following, we choose $l$ functions
$g_{1}, \ldots, g_{l}$ to be elements of 
${\cal G}(({\cal Y}\times {\cal X}^{2})_{\vec{W}})$ under Conditions {\bf (E1)}, {\bf (E2)}, and {\bf (E3)}.

For $ y_0\in {\cal Y}$, we choose $d$ functions $\bar{g}_{1,y_0}, \ldots, \bar{g}_{d,y_0}
\in {\cal G}(({\cal X}^{2})_{W_{y_0}})$
such that any non-zero linear combination of $\bar{g}_{1,y_0}, \ldots, \bar{g}_{d,y_0}$ does not belong to 
the $d^2-d$-dimensional space
\[ \{  W_{y_0}(x|x')^{-1} \alpha_{y_0}(A)(x|x')| ~A^T|u_{{\cal X}}\rangle=0 \},\]
and satisfies $\sum_{x,x'}P_{\vec{W}}(x')W_y(x|x')h(x|x')\neq 0$.
For $y_1 \in {\cal Y}$,
we choose $d^2-d$ functions $\bar{g}_{1,y_1}, \ldots, \bar{g}_{d^2-d,y_1}
\in {\cal G}(({\cal X}^{2})_{W_{y_1}})$ such that
any non-zero linear combination of $\bar{g}_{1,y_1}, \ldots, \bar{g}_{d^2-d,y_1}$ does not belong to the $d$-dimensional space
\[ \{  W_{y_1}(x|x')^{-1} \alpha_{y_1}(A)(x|x')+ c
|  A \in \Ker \alpha_{y_0} , ~c\in \mathbb{R}\}.\]
For remaining elements $y (\neq y_0,y_1)\in {\cal Y}$,
we choose $d^2$ functions $\bar{g}_{1,y}, \ldots, \bar{g}_{d^2,y}$
in $ {\cal G}(({\cal X}^2)_{W_y})$ such that $\sum_{x,x'}P_{\vec{W}}(x')W_y(x|x')\hat{g}_{j,y}(x|x')\neq 0$.
Hence, we have $d^2(d_Y-1)$ functions with the forms $g_{j,y}$, totally.

Then, we define $l=d^2(d_Y-1)$ functions $\hat{g}_{1}, \ldots, \hat{g}_{l}$ 
by renumbering the above $l$ functions $\bar{g}_{j,y}$ as follows.
We identify $y_0=0$ and $y_1=1$. Then, we define
\begin{align}
\hat{g}_{i+ (j-2)d^2 }(y,x,x'):=
\left\{
\begin{array}{ll}
\delta_{y_0,y} \bar{g}_{i,y_0}(x|x') & \hbox{ when } j=2, i=1, \ldots, d \\ 
\delta_{y_1,y} \bar{g}_{i-d,y_1}(x|x') & \hbox{ when } j=2, i=d+1, \ldots, d^2 \\ 
\delta_{j,y} \bar{g}_{i,j}(x|x') & \hbox{ when } j=3,\ldots, d_Y, i=1, \ldots, d^2 
\end{array}
\right. \Label{NFR}
\end{align}
for $i=1, \ldots, d^2$ and $j=2, \ldots, d_Y$.

When $\hat{g}_{j}$ is defined from $\bar{g}_{j',y}$, it is defined as 
$\hat{g}_{j}(y',x,x'):=\bar{g}_{j',y}(x|x')\delta_{y,y'}$.
This construction satisfies the following lemma.

\begin{lemma}\Label{T1-27}
The space spanned by the above given generators $\hat{g}_1, \ldots, \hat{g}_{l} $ 
has intersection $\{0\}$ with
${\cal N}(({\cal Y}\times {\cal X}^{2})_{\vec{W}})
+{\cal N}_{P_{\vec{W}}}(({\cal Y}\times {\cal X}^{2})_{\vec{W}})
+{\cal N}_2(({\cal Y}\times {\cal X}^{2})_{\vec{W}}))
=
{\cal N}_3(({\cal Y}\times {\cal X}^{2})_{\vec{W}})
+{\cal N}_{P_{\vec{W}}}(({\cal Y}\times {\cal X}^{2})_{\vec{W}})
$.
\end{lemma}

This lemma gives a canonical construction of generators of hidden Markov process
at non-singular points.


Next, we discuss what kinds of generators can be chosen by considering the linear combinations 
${g}_{1}, \ldots {g}_{l}$ of $\hat{g}_{1}, \ldots \hat{g}_{l}$.
That is, we consider elements ${g}_{1}, \ldots {g}_{l'}\in 
{\cal G}(({\cal Y}\times{\cal X}^{2})_{\vec{W}})\cap {\cal G}({\cal Y}_{\vec{W}})
={\cal G}({\cal Y}_{\vec{W}})$ 
and ${g}_{l'+1}, \ldots {g}_{l}\in 
{\cal G}(({\cal Y}\times{\cal X}^{2})_{\vec{W}})
\setminus {\cal G}({\cal Y}_{\vec{W}})$
such that
${g}_{1}, \ldots, {g}_{l'}$ are linearly independent of 
${g}_{l'}, \ldots, {g}_l$ and
any non-zero linear combination of ${g}_{l'}, \ldots, {g}_l$ 
is not contained in ${\cal G}({\cal Y}_{\vec{W}})$.
Hence, among $l$ generators, 
we can directly observe $l'$ generators at most.
Since the dimension of the quotient space generated by ${\cal G}({\cal Y})_{\vec{W}}$ is 
$d_Y-1 $, $l'$ is calculated as
\begin{align}
l':=d_Y-1. \Label{3-8-1}
\end{align}

\section{Two-hidden-state case in general model}\Label{s7}
\subsection{Two-state observation case}\Label{s7-1}
As the simplest example, we consider the case with $d=d_Y=2$. So, we denote ${\cal X}$ and ${\cal Y}$ by $\{0,1\}$.
We assume that the transition matrix $|\vec{W}|$ on ${\cal X}$ is irreducible and ergodic.
Moreover, all of the components of $|\vec{W}|$ are assumed to be strictly positive, i.e., Condition E2 holds.
In this case, we have
\begin{align}
{\cal N}
(({\cal Y} \times {\cal X}^2)_{\vec{W}})
=
\left\{
\left(
\left(
\begin{array}{cc}
a_1 & a_1+a_2 \\
a_1-a_2 & a_1
\end{array}
\right)
\right)_{y \in {\cal Y}}
\right\}_{(a_1,a_2)} .\Label{eq2-13-9B}
\end{align}
Since $\dim {\cal G}(({\cal Y}\times {\cal X}^{2})_{\vec{W}})=8$
and $\dim {\cal N}(({\cal Y}\times {\cal X}^{2})_{\vec{W}})= 2$,
we see that $\dim {\cal G}_1( ({\cal Y}\times {\cal X}^{2})_{\vec{W}})
=\dim {\cal L}_{1,\vec{W}}=6$.
In the following, we mainly discuss the tangent space with the $m$-representation.

\subsubsection{Non-singular points}
First, we assume that the relation
$P_{Y|X'}(0|0)=P_{Y|X'}(0|1)$, i.e., 
\begin{align}
W_0(0|0)+W_0(1|0)=W_0(0|1)+W_0(1|1) 
\Label{2-11-1}
\end{align}
does not hold.
This condition is equivalent to $P_{Y|X'=0} \neq P_{Y|X'=1}$.
So, E3 holds,
and we find that $k_{\vec{W}}=k_{(P_{\vec{W}},\vec{W})}=1$.
Then, $ \dim {\cal V}^{k_{(P,\vec{W})}} (P_{\vec{W}})=2$ 
and $\dim \Ker P^{k_{\vec{W}}}[\vec{W}] =0$, i.e., E2 hold.
So, we have $\dim {\cal L}_{P_{\vec{W}},\vec{W}}=0$ and
\begin{align*}
{\cal L}_{2,\vec{W}} 
=&
\left\{
\left(
\left[ W_0,
\left(
\begin{array}{cc}
a_1 & a_2 \\
-a_1 & -a_2
\end{array}
\right)
\right],
\left[ W_1,
\left(
\begin{array}{cc}
a_1 & a_2 \\
-a_1 & -a_2
\end{array}
\right)
\right]
\right)
\right\}_{(a_1,a_2)} .
\end{align*}
Further $\left[ W_0,
\left(
\begin{array}{cc}
a_1 & a_2 \\
-a_1 & -a_2
\end{array}
\right)
\right]$ is zero if and only if $a_1=a_2=0$,
which implies $\dim{\cal L}_{2,\vec{W}} =2$.
We apply the construction of generators given in Section \ref{s6-3}.
Thus, we notice that $l'=1$ and $l=3$.
That is, the dimension of the model is $3$.

The relation \eqref{2-11-1} holds if and only if the matrix 
$\left[ W_0,
\left(
\begin{array}{cc}
a_1 & a_2 \\
-a_1 & -a_2
\end{array}
\right)
\right]$ is given as a scalar times of 
$\left(
\begin{array}{cc}
1 & 1 \\
-1 & -1
\end{array}
\right)$ for any vector $(a_1,a_2)\in \mathbb{R}^2$.
The matrix $\left[ W_0,
\left(
\begin{array}{cc}
a_1 & a_2 \\
-a_1 & -a_2
\end{array}
\right)
\right]$ is traceless.
So, we can choose $\bar{g}_{1,0},\bar{g}_{2,0} \in {\cal G}(( {\cal X}^{2})_{W_{0}})$ as 
\begin{align}
W_{0*} (\bar{g}_{1,0}):=
\left(
\begin{array}{cc}
1 & 1 \\
-1 & -1
\end{array}
\right),\quad
W_{0*} (\bar{g}_{2,0}):=
\left(
\begin{array}{cc}
1 & 1 \\
1 & 1
\end{array}
\right).\Label{eq2-13-1}
\end{align}
Also, we can choose $\bar{g}_{1,1},\bar{g}_{2,1} \in {\cal G}(( {\cal X}^{2})_{W_{1}})$ as 
\begin{align}
W_{1*} (\bar{g}_{1,1}):=
\left(
\begin{array}{cc}
1 & 0 \\
-1 & 0
\end{array}
\right), \quad
W_{1*} (\bar{g}_{2,1}):=
\left(
\begin{array}{cc}
0 & 1 \\
0 & -1
\end{array}
\right).\Label{eq2-13-2}
\end{align}
So, we have the following descriptions of $
g_{1},g_{2},g_{3} ,g_{4} 
\in 
{\cal G}(({\cal Y}\times {\cal X}^{2})_{\vec{W}})
\setminus
(
{\cal N}_{3}(({\cal Y}\times {\cal X}^{2})_{\vec{W}})
+
{\cal N}_{P_{\vec{W}}}
(({\cal Y}\times {\cal X}^{2})_{\vec{W}}))$;
\begin{align}
g_{1}
=&
\left(
\left(
\begin{array}{cc}
W_0(0|0)^{-1} & W_0(0|1)^{-1} \\
-W_0(1|0)^{-1} & -W_0(1|1)^{-1}
\end{array}
\right),
\left(
\begin{array}{cc}
0 & 0 \\
0 & 0
\end{array}
\right)
\right),
\\
g_{2}
=&
\left(
\left(
\begin{array}{cc}
W_0(0|0)^{-1} & W_0(0|1)^{-1} \\
W_0(1|0)^{-1} & W_0(1|1)^{-1}
\end{array}
\right),
\left(
\begin{array}{cc}
0 & 0 \\
0 & 0
\end{array}
\right)
\right),
\\
g_{3}
=&
\left(
\left(
\begin{array}{cc}
0 & 0 \\
0 & 0
\end{array}
\right),
\left(
\begin{array}{cc}
W_1(0|0)^{-1} & 0 \\
-W_1(1|0)^{-1} & 0
\end{array}
\right)
\right),
\\
g_{4}
=&
\left(
\left(
\begin{array}{cc}
0 & 0 \\
0 & 0
\end{array}
\right),
\left(
\begin{array}{cc}
0 & W_1(0|1)^{-1}  \\
0 & -W_1(1|1)^{-1}
\end{array}
\right)
\right).
\end{align}
So, the four generators $ [g_{1}], [g_{2}], [g_{3}], [g_{4}]$ 
span the linear space
${\cal G}(({\cal Y}\times {\cal X}^{2})_{\vec{W}})
/(
{\cal N}_{3}(({\cal Y}\times {\cal X}^{2})_{\vec{W}})
+{\cal N}_{P_{\vec{W}}}(({\cal Y}\times {\cal X}^{2})_{\vec{W}}))$.


\subsubsection{Singular points}
Next, we assume that the relation \eqref{2-11-1} holds.
We find that $k_{\vec{W}}=1$ and $\Ker P^{k_{\vec{W}}}[\vec{W}] = \{(a,-a)^T\}$ and that
${\cal V}^k(P_{\vec{W}}) $ is the one-dimensional space spanned by $P_{\vec{W}}$.
Then, $k_{(P_{\vec{W}},\vec{W})}=1$,
$ \dim {\cal V}^{k_{(P,\vec{W})}}(P_{\vec{W}})=1$, and 
$\dim \Ker P^{k_{\vec{W}}}[\vec{W}]  =1$. 
Since the condition E1 nor E3 does not hold,
we need to construct the generators in a way different from the construction of generators given in Section \ref{s6-3}.
Then, we have 
\begin{align*}
{\cal L}_{2,\vec{W}} 
=&
\left\{
\left(
\left(
\begin{array}{cc}
a_1 & a_1 \\
-a_1 & -a_1
\end{array}
\right),
\left(
\begin{array}{cc}
a_1 & a_1 \\
-a_1 & -a_1
\end{array}
\right)
\right)
\right\}_{a_1} \\
{\cal L}_{P_{\vec{W}},\vec{W}} 
=&
\left\{
\left(
\left(
\begin{array}{cc}
a_1 & a_2 \\
-a_1 & -a_2
\end{array}
\right),
\left(
\begin{array}{cc}
a_3 & a_4 \\
-a_3 & -a_4
\end{array}
\right)
\right)
\right\}_{(a_1,a_2,a_3,a_4)},
\end{align*}
which implies that ${\cal L}_{P_{\vec{W}},\vec{W}} $
contains ${\cal L}_{2,\vec{W}} $.
Also,
\begin{align}
{\cal N}
(({\cal Y} \times {\cal X}^2)_{\vec{W}})
=
\left\{
\left(
\left(
\begin{array}{cc}
a_1 & a_1+a_2 \\
a_1-a_2 & a_1
\end{array}
\right),
\left(
\begin{array}{cc}
a_1 & a_1+a_2 \\
a_1-a_2 & a_1
\end{array}
\right)
\right)
\right\}_{(a_1,a_2)} .\Label{eq2-13-9}
\end{align}
So, we can choose elements 
$g_{1} \in {\cal G}({\cal Y})_{\vec{W}}
\setminus 
({\cal N}_{3}(({\cal Y}\times {\cal X}^{2})_{\vec{W}})
+{\cal N}_{P_{\vec{W}}}(({\cal Y}\times {\cal X}^{2})_{\vec{W}})))
$ 
and 
$g_{2} \in {\cal G}({\cal Y},{\cal X}^{2})_{\vec{W}}\setminus 
({\cal N}_{3}(({\cal Y}\times {\cal X}^{2})_{\vec{W}})
+{\cal N}_{P_{\vec{W}}}(({\cal Y}\times {\cal X}^{2})_{\vec{W}})))$ 
as 
\begin{align}
g_{1} :=
\left(E_1,-E_1 \right) ,\quad
g_{2} :=
\left(E_2,-E_2 \right) ,
\end{align}
where
\begin{align}
E_1:=
\left(
\begin{array}{cc}
1 & 1 \\
1 & 1 
\end{array}
\right),
\quad
E_2:=
\left(
\begin{array}{cc}
1 & -1 \\
1 & -1 
\end{array}
\right).
\end{align}
So, we have $l'=1$ and $l=2$.
That is, the local dimension at $\vec{W}$ is $2$.
The function $g_{1}$ expresses the variation inside of the set of singular points,
and the function $g_{2}$ expresses the variation in the direction orthogonal to the set of singular points.


\subsection{General case}\Label{s7-2}
Next, we consider the case when $d=2$ but $d_Y>2$. 
So, we denote ${\cal Y}$ by $\{0,1, \ldots, d_Y-1\}$.
Similarly, we assume that 
all of the components of $|\vec{W}|$ are assumed to be strictly positive, i.e., Condition E2 holds.
Hence, we have \eqref{eq2-13-9B}.
Since $\dim {\cal G}(({\cal Y}\times {\cal X}^{2})_{\vec{W}})=4 d_Y$
and $\dim {\cal N}(({\cal Y}\times {\cal X}^{2})_{\vec{W}})= 2$ ,
we see that $\dim {\cal G}_1( ({\cal Y}\times {\cal X}^{2})_{\vec{W}})
=\dim {\cal L}_{1,\vec{W}}=4 d_Y-2$.

\subsubsection{Non-singular points}
First, we assume that 
there exists an element $y_0 \in {\cal Y}$ such that
the relation $P_{Y|X'}(y_0|0)=P_{Y|X'}(y_0|1)$, i.e., 
\begin{align}
W_{y_0}(0|0)+W_{y_0}(1|0)=W_{y_0}(0|1)+W_{y_0}(1|1) 
\Label{2-11-1B}
\end{align}
does not hold.
So, we find that 
there exists another element $y_1 (\neq y_0) \in {\cal Y}$ such that
the relation \eqref{2-11-1B} does not hold.
So, E3 holds,
and we find that $k_{\vec{W}}=k_{(P_{\vec{W}},\vec{W})}=1$ and 
$\dim \Ker P^{k_{\vec{W}}}[\vec{W}] =0$, i.e., E2 holds.
Then, $ \dim {\cal V}^{k_{(P,\vec{W})}} (P_{\vec{W}})=2$. 
So, we have $\dim {\cal L}_{P_{\vec{W}},\vec{W}}=0$ and
\begin{align*}
{\cal L}_{2,\vec{W}} 
=&
\left\{
\left(
\left[ W_y,
\left(
\begin{array}{cc}
a_1 & a_2 \\
-a_1 & -a_2
\end{array}
\right)
\right]
\right)_{y}
\right\}_{(a_1,a_2)} .
\end{align*}
In the same way as Subsection \ref{s7-1}, we can show that $\dim{\cal L}_{2,\vec{W}} =2$.
We apply the construction of generators given in Section \ref{s6-3}.
So, we find that 
$l'=d_Y-1$ and $l=4 d_Y-2-2=4(d_Y-1)$.
That is, the dimension of the model is $4(d_Y-1)$.

In the same way as Subsection \ref{s7-1},
we can choose $\bar{g}_{1,y_0},\bar{g}_{2,y_0} \in {\cal G}(( {\cal X}^{2})_{W_{y_0}})$ and
$\bar{g}_{1,y_1},\bar{g}_{2,y_1} 
\in {\cal G}(( {\cal X}^{2})_{W_{y_1}})$ as \eqref{eq2-13-1} and \eqref{eq2-13-2},
respectively.
For other elements $y (\neq y_0,y_1) \in {\cal Y}$,
we can choose 
$\bar{g}_{1,y},\bar{g}_{2,y}, \bar{g}_{3,y},\bar{g}_{4,y} 
\in {\cal G}(( {\cal X}^{2})_{W_{y}})$ as
\begin{align}
W_{y,*} (\bar{g}_{1,y})
=&
\left(
\begin{array}{cc}
W_y(0|0) & 0 \\
0 & 0
\end{array}
\right), \quad
W_{y,*} (\bar{g}_{2,y})
=
\left(
\begin{array}{cc}
0 & 0 \\
W_y(1|0) & 0
\end{array}
\right), \\
W_{y,*} (\bar{g}_{3,y})
=&
\left(
\begin{array}{cc}
0 & W_y(0|1) \\
0 & 0
\end{array}
\right), \quad
W_{y,*} (\bar{g}_{4,y})
=
\left(
\begin{array}{cc}
0 & 0 \\
0 & W_y(1|1)
\end{array}
\right).
\end{align}
Using the same method as \eqref{NFR} in Section \ref{s6-3}, we define 
the three generators $g_{1},\ldots, g_{4(d_Y-1)}$, which span the linear space
${\cal G}(({\cal Y}\times {\cal X}^{2})_{\vec{W}})
/(
{\cal N}_{3}(({\cal Y}\times {\cal X}^{2})_{\vec{W}})
+{\cal N}_{P_{\vec{W}}}(({\cal Y}\times {\cal X}^{2})_{\vec{W}}))$.


\subsubsection{Singular points}
Next, we assume that the relation \eqref{2-11-1B} holds for all points $y \in {\cal Y}$.
We find that $k_{\vec{W}}=1$ and $\Ker P^{k_{\vec{W}}}[\vec{W}] = \{(a,-a)^T\}$ and that
${\cal V}^k(P_{\vec{W}}) $ is the one-dimensional space spanned by $P_{\vec{W}}$.
Then, $k_{(P_{\vec{W}},\vec{W})}=1$,
$ \dim {\cal V}^{k_{(P,\vec{W})}}(P_{\vec{W}})=1$, and 
$\dim \Ker P^{k_{\vec{W}}}[\vec{W}]  =1$. 
Since the condition E1 nor E3 does not hold,
we need to construct the generators in a way different from the construction of generators given in Section \ref{s6-3}.
Then, we have 
\begin{align*}
{\cal L}_{2,\vec{W}} 
=&
\left\{
\left(
\left(
\begin{array}{cc}
a_1 & a_1 \\
-a_1 & -a_1
\end{array}
\right)
\right)_{y \in {\cal Y}}
\right\}_{a_1} \\
{\cal L}_{P_{\vec{W}},\vec{W}} 
=&
\left\{
\left(
\left(
\begin{array}{cc}
a_{1,y} & a_{2,y} \\
-a_{1,y} & -a_{2,y}
\end{array}
\right)
\right)_{y \in {\cal Y}}
\right\}_{(a_{1,y},a_{2,y})},
\end{align*}
which implies that ${\cal L}_{P_{\vec{W}},\vec{W}} $
contains ${\cal L}_{2,\vec{W}} $.
Hence, we find that $l= 4d_y-2 - 2 d_Y= 2d_Y-2$.
So, we can choose elements 
$g_{1},\ldots,g_{d_Y-1} \in {\cal G}({\cal Y})_{\vec{W}}
\setminus 
({\cal N}_{3}(({\cal Y}\times {\cal X}^{2})_{\vec{W}})
+{\cal N}_{P_{\vec{W}}}(({\cal Y}\times {\cal X}^{2})_{\vec{W}})))
$ 
and 
$g_{d_Y},\ldots,g_{2d_Y-2} \in {\cal G}({\cal Y},{\cal X}^{2})_{\vec{W}}\setminus 
({\cal N}_{3}(({\cal Y}\times {\cal X}^{2})_{\vec{W}})
+{\cal N}_{P_{\vec{W}}}(({\cal Y}\times {\cal X}^{2})_{\vec{W}})))$ 
as 
\begin{align}
g_{j} :=
\left( b_y^j E_1 \right)_{y \in {\cal Y}} ,\quad
g_{d_Y-1+j} :=
\left( b_y^j E_2 \right)_{y \in {\cal Y}} 
\end{align}
for $j=1, \ldots, d_Y-1$,
where 
the vectors
$(b_y^j)_{y \in {\cal Y}} $ $(j=1, \ldots, d_Y-1)$
are linearly independent.
 So, we have $l'=d_Y-1$ because $l'\le d_Y-1$.

Since the set of singular points are given as the set of points satisfying the condition \eqref{2-11-1B} for any $y \in {\cal Y}$,
the functions $g_{1},\ldots,g_{d_Y-1}$ express the variations inside of the set of singular points,
and the functions $g_{d_Y},\ldots,g_{2d_Y-2}$ express the variations of the direction orthogonal to the set of singular points.


\section{Conditionally independent case}\Label{s8}
\subsection{Equivalence problem}\Label{s8-1}
Sections \ref{s5}-\ref{s7} discussed the case when $X_n$ and $Y_n$ are correlated even with a fixed value $X_{n-1}=x_{n-1}$.
Now, we consider the special case when $X_n$ and $Y_n$ are conditionally independent with a fixed value $X_{n-1}=x_{n-1}$,
which is illustrated in Fig. \ref{2model}.
In this case, 
the ${\cal Y}$-indexed transition matrix $\vec{W}= (W_y(x|x'))_{y \in {\cal Y}}$ is given as 
$W_y(x|x')=W(x|x') V_y(x')= W D(V_y)(x|x')$
where $W(x|x')$ is a transition matrix on ${\cal X}$, 
$V(y|x')$ is a transition matrix with the input ${\cal X}$ and the output ${\cal Y}$,
$V_{y}$ is the vector satisfying $V_y(x')=V(y|x')$,
and $D(v)$ is the diagonal matrix whose diagonal entries are given by a vector $v$.
We call the above type of ${\cal Y}$-indexed transition matrix 
an independent-type ${\cal Y}$-indexed transition matrix, and denote it by 
$(W,V)$.
Also, we define the vector ${V}_{*,x'} $ as ${V}_{*,x'}(y):= V(y|x')$ for 
an independent-type ${\cal Y}$-indexed transition matrix $(W,V)$.

Here, we rewrite the notations defined in Subsection \ref{s5-1} by using the pair of transition matrices $W,V$.
The transition matrix  $P^k[(W,V)] $ is given as
\begin{align}
P^k[(W,V)] (y_k, \ldots, y_1|x')=&
\sum_{x\in {\cal X}}W D(V_{y_k}) W \cdots W D(V_{y_1})(x|x') \nonumber \\
=&
\sum_{x\in {\cal X}} D(V_{y_k}) W \cdots W D(V_{y_1})(x|x').\Label{LEV2}
\end{align}
The integer $k_{(W,V)}$ is the minimum integer $k_0$ to satisfy the condition
$\Ker P^{k_0}[(W,V)]=\cap_{k} \Ker P^{k}[(W,V)]$. 
For a distribution $P$ on ${\cal X}$, the subspace ${\cal V}^k(P)$ 
is the subspace of ${\cal V}_{{\cal X}}/ \Ker P^{k_{(W,V)}}[(W,V)]$
spanned by \par
\noindent$\{ [ W D(V_{y_k}) W \cdots W D(V_{y_1})P] | y_j \in \cY, k'\le k\}$,
where $[v]$ expresses the element of 
${\cal V}_{{\cal X}}/ \Ker P^{k_{(W,V)}}[(W,V)]$ whose 
representative is $v \in {\cal V}_{{\cal X}}$.
Then, the integer $ k_{(P,(W,V))}$ is the minimum integer $k_1$ 
to satisfy the condition
$\cup_{k =1}^{\infty} {\cal V}^k(P) = {\cal V}^{k_1}(P)$.

The kernel $\Ker P[(W,V)]$ is characterized as follows.

\begin{lemma}\Label{KHG}
Given an independent-type ${\cal Y}$-indexed transition matrix $(W,V)$ on ${\cal X}$ and a distribution $P$ on ${\cal X}$, 
we assume that the vectors $\{{V}_{*,x'}\}_{x' \in {\cal X}} $ are linearly independent.
Then, $\Ker P^1[(W,V)]=\{0\}$ and $k_{(W,V)}=1 $.
\end{lemma}

\begin{proof}
Since $\{{V}_{*,x'}\}_{x' \in {\cal X}} $ are linearly independent,
the rank of the matrix $ P^1[(W,V)] (y|x')=V(y|x')$ is $d$.
Hence, $\Ker P^1[(W,V)]=\{0\}$, which implies the relation $k_{(W,V)}=1 $.
\hfill$\Box$\end{proof}

Under a similar condition, 
the equivalent conditions are characterized as follows.

\begin{lemma}\Label{KHF}
Given an independent-type ${\cal Y}$-indexed transition matrix $(W,V)$ on ${\cal X}$ and a distribution $P$ on ${\cal X}$, 
we assume that the vectors ${V}_{*,x'} $ are linearly independent
and the support of $P$ is ${\cal X} $. 
When the pair of an independent-type ${\cal Y}$-indexed transition matrix $(W',V')$ 
and a distribution $P'$ 
is equivalent to the pair of the 
independent-type ${\cal Y}$-indexed transition matrix $(W,V)$ on ${\cal X}$ and a distribution $P$ on ${\cal X}$, 
there exists a permutation $g$ among the elements of ${\cal X}$ such that
$W'= g^{-1} W g $, $V'=V g$, and $ P'=P g $.
\end{lemma}

This lemma shows that 
the above assumption guarantees that
there is no equivalent pair of 
an independent-type ${\cal Y}$-indexed transition matrix and a distribution 
except for a permuted one.

\begin{proof}
Since the vectors ${V}_{*,x'} $ are linearly independent, we find that
$\Ker P^1[(W,V)]=\{0\}$.
There exists a linear map $T$ on ${\cal V}_{{\cal X}}$ such that
\begin{align}
T W D(V_y) = W' D(V_y') T
\end{align}
for any $y \in {\cal Y}$.
Since
\begin{align}
T W =\sum_y T W D(V_y) =\sum_y W' D(V_y') T=W' T, 
\end{align}
we have
\begin{align}
T W T^{-1} T D(V_y) 
= W' T T^{-1}D(V_y') T
= T W T^{-1}D(V_y') T,
\end{align}
which implies that
\begin{align}
T D(V_y) = D(V_y') T.
\end{align}
Hence, $T$ is a permutation on ${\cal X}$, which yields the desired statement.
\hfill$\Box$\end{proof}

Although we introduce independent-type ${\cal Y}$-indexed transition matrices,
it is not so trivial to clarify whether a given ${\cal Y}$-indexed transition matrix  is equivalent to
an independent-type ${\cal Y}$-indexed transition matrix.
The following lemma answers this question.

\begin{lemma}\Label{HGR}
The following conditions are equivalent for a ${\cal Y}$-indexed transition matrix 
 $\vec{W}= (W_y(x|x'))_{y \in {\cal Y}}$
 when $|\vec{W} |$ is invertible.
\begin{description}
\item[\bf (G1)]
There exists 
an independent-type ${\cal Y}$-indexed transition matrix $(W,V)$ equivalent to the ${\cal Y}$-indexed transition matrix $\vec{W}$.

\item[\bf (G2)] The following three conditions hold.
\begin{description}
\item[\bf (G2-1)]
The characteristic polynomial $U_y$ has no multiple root,
and the eigenvalues of $U_y$ are non-negative real numbers, where $U_y:= |\vec{W}|^{-1}W_y$.
\item[\bf (G2-2)]
The matrices $\{U_y\}_{y \in {\cal Y}}$
have a common eigenvector system $\{t_i\}$, where $t_i$ is normalized so that $\langle u_X| t_i\rangle=1$.
\item[\bf (G2-3)]
The matrix $ T |\vec{W}| T^{-1}$ has non-negative entries, where the matrix $T$ is given as $T(i|x):=t_i(x)$.
\end{description}
\end{description}
\end{lemma}

\begin{proof}
Assume ${\bf (G1)}$.
Then,  there exists a matrix $T$ on ${\cal V}_{{\cal X}}$ such that
\begin{align}
T W_y &= W D(V_y) T \Label{3-9-4} \\ 
T^T|u_{{\cal X}}\rangle &=|u_{{\cal X}}\rangle.\Label{3-9-1}
\end{align}
Thus,
\begin{align}
W T= \sum_{y}WD(V_y) T 
=\sum_{y}WD(V_y) T 
=\sum_{y}T W_y=T |\vec{W}|. \Label{2-9-3}
\end{align}
Hence, 
\begin{align}
W D(V_y) T
=T W_y
= T |\vec{W}| U_y
=W T U_y.
\end{align}
Since $|\vec{W} |$ is invertible, 
\begin{align}
D(V_y) T=T U_y,\Label{2-9-2}
\end{align}
which implies {\bf (G2-1)}.
We choose $t_i$ as the $i$-th row matrix of $T$.
The relations \eqref{3-9-1} and \eqref{2-9-2} guarantee {\bf (G2-2)}.
Then, \eqref{2-9-3} implies {\bf (G2-3)}.

Assume {\bf (G2)}.
{\bf (G2-2)} guarantees \eqref{3-9-1}. Hence, {\bf (G2-3)} guarantees that 
the matrix $W:= T |\vec{W}| T^{-1}$ is a probability transition matrix.
Due to {\bf (G2-1)} and {\bf (G2-2)},
we can choose the vector $V_y$ to satisfy \eqref{2-9-2}.
So, we obtain \eqref{3-9-4}.
Thus, we obtain {\bf (G1)}.
\hfill$\Box$\end{proof}

It is not so easy to satisfy the condition {\bf (G2)}.
However, when $|{\cal Y}|=2$, it is not so difficult to satisfy the condition {\bf (G2)}.
In this case, 
once {\bf (G2-1)} is satisfied, {\bf (G2-2)} is automatically satisfied.

Although Lemma \ref{KHG} guarantees the relation 
$\Ker P^{k_{(W,V)}}[(W,V)]=\{0\}$ under a certain condition for the transition matrix $V$,
the condition is too strong because it does not hold when $d_Y < d$.
Even when $d_Y < d$, we can expect the relations
$\Ker P^{k_{(W,V)}}[(W,V)]=\{0\}$ and ${\cal V}^{k_{P,(W,V)}}(P)={\cal V}_{{\cal X}}$
under some natural condition.
The following lemma shows how frequently these conditions hold.

\begin{lemma}\Label{2-15-B}
We fix a transition matrix $V$, and assume the existence of $y \in {\cal Y}$ such that $ V_y$ is not a scalar times of $u_{{\cal X}}$.
The relations
$\Ker P^{k_{(W,V)}}[(W,V)]=\{0\}$ and ${\cal V}^{k_{P,(W,V)}}(P)={\cal V}_{{\cal X}}$
hold almost everywhere with respect to $W$ and $P$.
Also, the relations $\Ker P^{k_{(W,V)}}[(W,V)]=\{0\}$ and ${\cal V}^{k_{P,(W,V)}}(P_{W})={\cal V}_{{\cal X}}$
hold almost everywhere with respect to $W$.
\end{lemma}

\begin{proof}
We show the desired statement when $P$ and $P_{W}=P'$ are fixed
and we impose the condition $WP'=P'$, 
which is sufficient for both statements.
We fix $y \in {\cal Y}$ such that $ V_y$ is not a scalar times of $u_{{\cal X}}$.
We choose $k$ to be $d-1$, and 
choose linearly independent $k+1$ vectors $|v_0\rangle,\ldots,| v_k\rangle$ 
in the dual space of ${\cal V}_{{\cal X}}$
such that
$|v_0\rangle$ is a scalar times of $|u_{{\cal X}}\rangle$
and $\langle v_0| P'\rangle=1$.
We choose $a_{i,j}$ such that 
$ W^T|v_j\rangle=\sum_{i=0}^k a_{i,j} |v_i\rangle $.
The condition $WP'=P'$ is equivalent to the condition
$\sum_{i=0}^k a_{i,j}=1 $, i.e., $a_{0,j}=1-\sum_{i=1}^k a_{i,j}$.
Hence, we can freely choose the coefficients $a_{i,j}$ for $1\le i,j \le k$
with the constraint that $W$ is a positive matrix.
Hence, the $k$ vectors
$D(V_y)|u_{\cal X}\rangle , W_y^TD(V_y) |u_{\cal X}\rangle,  
\ldots, (W_y^T)^{k-1} D(V_y)|u_{\cal X}\rangle $
are linearly independent as elements of the quotient space ${\cal V}_{{\cal X}}/<|u_{\cal X}\rangle> $
almost everywhere with respect to the above choice of $W$,
where $<|u_{\cal X}\rangle>$
is the one-dimensional space spanned by $|u_{\cal X}\rangle$.
Since $W^T |u_{\cal X}\rangle= |u_{\cal X}\rangle $, 
the $d$ vectors
$D(V_y)W^T|u_{\cal X}\rangle , W_y^TD(V_y) W^T|u_{\cal X}\rangle,  
\ldots$, $(W_y^T)^{k-1} D(V_y)W^T |u_{\cal X}\rangle $,
 $|u_{\cal X}\rangle$ spans the dual space of ${\cal V}_{{\cal X}}$, which implies 
the relation
$\Ker P^{k}[(W,V)]=\{0\}$, i.e.,
$\Ker P^{k_{(W,V)}}[(W,V)]=\{0\}$.

Next, we choose $y \in {\cal Y}$ such that 
$D(V_y) P $ is not a scalar times of $P'$.
If the above choice of $y \in {\cal Y}$ does not satisfy this condition, we choose another
$y \in {\cal Y}$
such that $D(V_y) P $ is not a scalar times of $P'$ because 
$\sum_{y \in{\cal Y}} D(V_y)=I$.
We replace the roles of $P'$ and $|u_{{\cal X}}$ in the above discussion.
Hence, the $k$ vectors
$D(V_y) P  , W_y D(V_y) P,  
\ldots, W_y^{k-1} D(V_y) P $
are linearly independent as elements of the quotient space ${\cal V}_{{\cal X}}/<P'> $
almost everywhere with respect to the above choice of $W$.
Further, $W^{k'}P$ is close to $P'$ when $k'$ is sufficiently large.
Hence, ${\cal V}^{k'}(P)={\cal V}_{{\cal X}}$, i.e.,
${\cal V}^{k_{P,(W,V)}}(P)={\cal V}_{{\cal X}}$.
\hfill$\Box$\end{proof}

\subsection{Exponential family}
Next, to give a suitable parametrization,
we consider the exponential family of independent-type ${\cal Y}$-indexed transition matrices.
Firstly, we fix an irreducible independent-type ${\cal Y}$-transition matrix 
$(W,V)$ on ${\cal X}$.
Then, we denote the support of $(W,V)$ by
$({\cal X}^{2}\cup {\cal Y}\times {\cal X})_{(W,V)}:=
{\cal X}^{2}_{W} \cup ({\cal Y}\times {\cal X})_V$.
Then, we denote the linear space of real-indexed functions 
$g=(g_a(x,x'),g_b(y,x'))$ defined on $({\cal X}^{2}\cup {\cal Y}\times {\cal X})_{(W,V)}$ 
by ${\cal G}(({\cal X}^{2}\cup {\cal Y}\times {\cal X})_{(W,V)})$.
Here, for an element $(x,x')\in {\cal X}^{2}$, the function is given as $g_a(x,x')$,
and for an element $(y,x')\in {\cal Y}\times {\cal X}$, the function is given as $g_b(y,x')$.
Now, we denote the ${\cal Y}$-transition matrix given by $(W,V)$,
by $\vec{W}$.
Then, $(g_a(x,x'),g_b(y,x'))$ is identified with the function $(x,x',y)\mapsto g_a(x,x')+g_b(y,x')$,
which is an element of 
${\cal G}( ({\cal Y}\times {\cal X}^{2})_{\vec{W}})$.
However, using a function $\bar{f}(x')$, 
we introduce other functions
$(\bar{g}_a(x,x'),\bar{g}_b(y,x'))$
as $\bar{g}_a(x,x'):= g_a(x,x')- \bar{f}(x')$
and $\bar{g}_b(y,x')):= g_b(y,x')+\bar{f}(x')$.
Then, the other pair of function $(\bar{g}_a(x,x'),\bar{g}_b(y,x'))$
corresponds to the same element of
${\cal G}( ({\cal Y}\times {\cal X}^{2})_{\vec{W}})$ as $(g_a(x,x'),g_b(y,x'))$.
To avoid this problem, we impose the condition $ \sum_{y \in {\cal Y}}V(y|x')g_b(y,x')=0$ for $x' \in 
{\cal X}$.
Hence, we denote the linear space of real-indexed functions 
$g=(g_a(x,x'),g_b(y,x'))$ defined on $({\cal X}^{2}\cup {\cal Y}\times {\cal X})_{(W,V)}$ with this constraint
by ${\cal G}_0(({\cal X}^{2}\cup {\cal Y}\times {\cal X})_{(W,V)})$.
Hence,
the space ${\cal G}(({\cal X}^{2}\cup {\cal Y}\times {\cal X})_{(W,V)})$
can be regarded as a subspace of ${\cal G}( ({\cal Y}\times {\cal X}^{2})_{\vec{W}})$.
Additionally, 
the subspace ${\cal N}^I(({\cal X}^{2}\cup {\cal Y}\times {\cal X})_{(W,V)})
:={\cal N}(({\cal Y}\times {\cal X}^{2})_{\vec{W}}) \cap
{\cal G}(({\cal X}^{2}\cup {\cal Y}\times {\cal X})_{(W,V)})$
equals the subspace ${\cal N}(({\cal Y}\times {\cal X}^{2})_{\vec{W}}) $, which is 
composed of functions with form $ f(x)-f(x')+c$.
That is, an element $g=(g_a(x,x'),g_b(y,x'))$ of the subspace ${\cal N}^I(({\cal X}^{2}\cup {\cal Y}\times {\cal X})_{(W,V)})$ has the form $g_a(x,x')= f(x)-f(x')+c $ and $g_b(y,x')=0$.

To give the relation between the $e$-representation and the $m$-representation,
we define the linear map $(W,V)_*$ on 
${\cal G}(({\cal X}^{2}\cup {\cal Y}\times {\cal X})_{(W,V)})$ as
\begin{align}
((W,V)_*g)_a(x,x'):=& g_a(x,x') W(x|x') \\
((W,V)_*g)_b(y,x'):=& g_b(y,x') V(y|x')
\end{align}
for $g\in {\cal G}(({\cal X}^{2}\cup {\cal Y}\times {\cal X})_{(W,V)})$.
To discuss the relation between the $m$-representations 
of the independent-type and the general case, 
we define the linear map $(W,V)^{*}$ 
from ${\cal G}( ({\cal X}^{2}\cup {\cal Y}\times {\cal X})_{(W,V)})$
to ${\cal G}(({\cal Y}\times {\cal X}^{2})_{(W,V)})$ as
\begin{align}
((W,V)^{*}g)(y,x,x'):= g_a(x,x')V(y|x')+W(x|x')g_b(y,x')
\end{align}
for $g\in {\cal G}( ({\cal X}^{2}\cup {\cal Y}\times {\cal X})_{(W,V)})$.
In the following, the function 
$g_a(x,x')$ is written as a matrix $B$ on ${\cal V}_{{\cal X}}$,
and 
$g_b(y,x')$ is written as a collection of vectors $(C_y)_{y}$, which belong to 
${\cal V}_{{\cal X}}$.
That is, the map $(W,V)^{*}$ is rewritten as 
\begin{align}
((W,V)^{*}(B,C))_y= B D(V_y)+WD(C_y).
\end{align}
Hence, when $(B,C) \in (W,V)^{*} {\cal G}_0(({\cal X}^{2}\cup {\cal Y}\times {\cal X})_{(W,V)})$
satisfies $\sum_y C_y=0$.

Define 
\begin{align}
{\cal G}_1( ({\cal X}^{2}\cup {\cal Y}\times {\cal X})_{(W,V)})
&:={\cal G}_1(({\cal Y}\times {\cal X}^{2})_{\vec{W}}) 
\cap {\cal G}_0(({\cal X}^{2}\cup {\cal Y}\times {\cal X})_{(W,V)}) \\
{\cal L}_{1,W,V}^I.
&:= (W,V)_*{\cal G}_1( ({\cal X}^{2}\cup {\cal Y}\times {\cal X})_{(W,V)}).
\end{align}

Then, we have the following lemma.
\begin{lemma}\Label{LRN}
The following relation holds;
\begin{align}
&{\cal L}_{1,W,V}^I =
\Big\{ 
(B,C)\in {\cal G}(({\cal X}^{2}\cup {\cal Y}\times {\cal X})_{(W,V)})
 \Big|
B^T  | u_{{\cal X}}\rangle =0,
\sum_{y \in {\cal Y}} C_{y}=0
\Big \} .
\end{align}
\end{lemma}

\begin{proof}
For $(B,C) \in {\cal G}(({\cal X}^{2}\cup {\cal Y}\times {\cal X})_{(W,V)}) $,
$(B,C) \in {\cal L}_{1,W,V}^I $ if and only if
\begin{align}
\sum_{y \in {\cal Y}} C_{y} &=0, \quad
(\sum_y((W,V)^{*}(B,C))_y)^T  | u_{{\cal X}}\rangle 
=0.
\end{align}
Since $W ^T| u_{{\cal X}}\rangle 
=| u_{{\cal X}}\rangle $,
we have
\begin{align}
&(\sum_y((W,V)^{*}(B,C))_y)^T  | u_{{\cal X}}\rangle 
=(\sum_y B D(V_y)+WD(C_y))^T
| u_{{\cal X}}\rangle \nonumber \\
=&( B+W \sum_y D(C_y))^T
| u_{{\cal X}}\rangle 
=( B^T+\sum_y D(C_y) W ^T)
| u_{{\cal X}}\rangle \nonumber \\
=&( B^T+\sum_y D(C_y))
| u_{{\cal X}}\rangle 
=B^T  | u_{{\cal X}}\rangle+
\sum_{y \in {\cal Y}} C_{y}.
\end{align}
Hence, we obtain the desired statement.
\end{proof}

The space ${\cal G}_1( ({\cal X}^{2}\cup {\cal Y}\times {\cal X})_{(W,V)})$
equals the space
${\cal G}_1(({\cal Y}\times {\cal X}^{2})_{(W,V)}) \cap
{\cal G}(({\cal X}^{2}\cup {\cal Y}\times {\cal X})_{(W,V)})$.
The space $ {\cal L}_{1,W,V}^I $  equals the space 
$(W,V)^{*}{\cal L}_{1,(W,V)}\cap
(W,V)^{*}{\cal G}(({\cal X}^{2}\cup {\cal Y}\times {\cal X})_{(W,V)})$.

Assume that functions $g_1, \ldots, g_l \in {\cal G}(({\cal X}^{2}\cup {\cal Y}\times {\cal X})_{(W,V)})$
are linearly independent as elements of ${\cal G}(({\cal X}^{2}\cup {\cal Y}\times {\cal X})_{(W,V)})/{\cal N}
(({\cal X}^{2}\cup {\cal Y}\times {\cal X})_{(W,V)})$
for $\vec{\theta}:=(\theta^1, \ldots, \theta^l) \in \bR^l$.
We define the transition matrix
$$ V_{\vec{\theta}}(y|x'):=
e^{\sum_{j=1}^l \theta^j g_{j,b}(y,x')} V(y|x')
/\sum_{y'}e^{\sum_{j=1}^l \theta^j g_{j,b}(y',x')} V(y'|x').$$
That is, for each $x' \in {\cal X}$, $ V_{\vec{\theta}}(y|x')$ forms an exponential family of distributions on ${\cal Y}$.
Also, we define the matrix
$$\overline{W}_{\vec{\theta}}(x|x')
:=\sum_{y} e^{\sum_{j=1}^l \theta^j (g_{j,a}(x, x') + g_{j,b}(y,x'))}V(y|x') W(x|x'),$$
and denote its Perron-Frobenius eigenvalue by $\lambda_{\vec{\theta}}$.
Also, we denote the Perron-Frobenius eigenvector of 
the transpose $\overline{W}_{\vec{\theta}}^T$
by $\overline{P}^3_{\vec{\theta}}$.
Then, we define the transition matrix
$W_{\vec{\theta}}(x| x'):=
\lambda_{\vec{\theta}}^{-1} \overline{P}^3_{\vec{\theta}}(x)
\overline{W}_{\vec{\theta}}(x|x') \overline{P}^3_{\vec{\theta}}(x')^{-1}$
on ${\cal X}$.
The ${\cal Y}$-indexed transition matrix 
generated by $g_1, \ldots, g_l$ is given as
\begin{align}
&W_{\vec{\theta},y}(x|x')
=
\lambda_{\vec{\theta}}^{-1} \overline{P}^3_{\vec{\theta}}(x)
e^{\sum_{j=1}^l \theta^j (g_{j,a}(x, x') + g_{j,b}(y,x'))}V(y|x') W(x|x')
\overline{P}^3_{\vec{\theta}}(x')^{-1} \nonumber \\
=&
\lambda_{\vec{\theta}}^{-1} \overline{P}^3_{\vec{\theta}}(x)
\overline{W}_{\vec{\theta}}(x|x') \overline{P}^3_{\vec{\theta}}(x')^{-1}
 V_{\vec{\theta}}(y|x')
=W_{\vec{\theta}}(x| x') V_{\vec{\theta}}(y|x').
\end{align}
That is, 
the family $(W_{\vec{\theta}},V_{\vec{\theta}})$ coincides with 
the exponential family of 
${\cal Y}$-indexed transition matrices on ${\cal X}$
generated by $g_1, \ldots, g_l $.
Hence, the family $(W_{\vec{\theta}},V_{\vec{\theta}})$ is called
an exponential family of independent-type ${\cal Y}$-indexed transition matrices.
Since 
an exponential family of 
${\cal Y}$-indexed transition matrices is a special case of 
an exponential family of transition matrices on ${\cal X}\times {\cal Y}$,
an exponential family of independent-type 
${\cal Y}$-indexed transition matrices is a special case of 
an exponential family of transition matrices on ${\cal X}\times {\cal Y}$.

\begin{example}\Label{ERT}
As an example, we consider 
the full parameter model of independent-type ${\cal Y}$-indexed transition matrices
on ${\cal X}$.
That is, we assume that the support
 $({\cal X}^{2}\cup {\cal Y}\times {\cal X})_{(W,V)}$ is ${\cal X}^{2}\cup {\cal Y}\times {\cal X}$
 and $W$ is irreducible.
The tangent space of the model is given by the space ${\cal L}_{1,W,V}^I $, whose dimension is 
$l:=d^2-d+d d_Y-d=d(d+d_Y-2)$.
In this case, we can easily find the generators as follows.
Here, we do not necessarily choose the generators from 
${\cal G}_1( ({\cal X}^{2}\cup {\cal Y}\times {\cal X})_{(W,V)})$.
That is, it is sufficient to choose them as elements of 
${\cal G}( ({\cal X}^{2}\cup {\cal Y}\times {\cal X})_{(W,V)})$.
For simplicity, we assume that ${\cal X}=\{1,\ldots, d\}$ and ${\cal Y}=\{1,\ldots, d_Y\}$.
We choose the functions 
$g_{j,a}$ and $g_{j,b}$ for $ 1\le j \le d_Y-1$, 
the functions 
$g_{i(d_Y-1)+j,a}$ and 
$g_{i(d_Y-1)+j,b}$ for $1\le i \le d-1, 1\le j \le d_Y-1$,
and the functions $g_{d(d_Y-1)+(i-1)(d-1)+j,a}$ and $g_{d(d_Y-1)+(i-1)(d-1)+j,b}$ 
for $1\le i \le d, 1\le j \le d-1$ as
\begin{align}
g_{j,a}(x,x') &:= 0 \Label{FR1}\\
g_{j,b}(y,x') &:=\delta_{y,j} \Label{FR2}\\
g_{i(d_Y-1)+j,a}(x,x') &:= 0 \Label{FR3}\\
g_{i(d_Y-1)+j,b}(y,x') &:= \delta_{x',i}\delta_{y,j}\Label{FR4} \\
g_{d(d_Y-1)+(i-1)(d-1)+j,a}(x,x') &:= \delta_{x',i}\delta_{x,j} \Label{FR5}\\
g_{d(d_Y-1)+(i-1)(d-1)+j,b}(y,x') &:= 0 .\Label{FR6}
\end{align}
Then, the functions $g_i'=(g_{i',a},g_{i',b})$ are linearly independent.
We can parametrize the full model of independent-type ${\cal Y}$-indexed transition matrices
by using this generators.
In particular, the first $d_Y-1$ functions belong to ${\cal G}({\cal Y}_{\vec{W}})$.
That is, the maximum number $l'$ of observed generators is $d_Y-1$ similar to \eqref{3-8-1}
because this number is upper bounded by $d_Y-1$.
\end{example}

Then, we obtain 
the exponential family of independent-type ${\cal Y}$-indexed transition matrices 
$\{(W_{\vec{\theta}},V_{\vec{\theta}})\}_{\vec{\theta}\in \bR^{l}}$,
which is generated by the above generators $g_{1}, \ldots, g_{l}$
at $(W,V)$.
While the set $\{(W_{\vec{\theta}},V_{\vec{\theta}})\}_{\vec{\theta}\in \bR^{l}}$
contains elements equivalent to each other,
we have the following lemma.
\begin{lemma}\Label{LKD}
When the independent-type ${\cal Y}$-indexed transition matrices $(W,V)$ satisfies Condition
$({\cal X}^{2}\cup {\cal Y}\times {\cal X})_{(W,V)}={\cal X}^{2}\cup {\cal Y}\times {\cal X}$,
the above defined set $\{(W_{\vec{\theta}},V_{\vec{\theta}}) \}_{\vec{\theta}}$ equals the set of
independent-type ${\cal Y}$-indexed transition matrices $(W',V')$ on ${\cal X}$ satisfying the relation
 $({\cal X}^{2}\cup {\cal Y}\times {\cal X})_{(W',V')}={\cal X}^{2}\cup {\cal Y}\times {\cal X}$.
\end{lemma}

\begin{proof}
When we freely choose the parameters $\theta_1, \ldots, \theta_{d (d_Y-1)}$,
the set  of $V_{\vec{\theta}}$ equals the set of transition matrices from ${\cal X}$ to ${\cal Y}$ with full support.
Next, we fix the parameters $\theta_1, \ldots, \theta_{d (d_Y-1)}$ and freely choose the remaining parameters
$\theta_{d (d_Y-1)+1}, \ldots, \theta_{d (d_Y-1)+d^2-d}$.
Then, the set $\{W_{\vec{\theta}}\}$ forms the exponential family generated by 
$g_{d (d_Y-1)+1,a}, \ldots, g_{d (d_Y-1)+d^2-d,a} $ at 
$W_{\theta_1, \ldots, \theta_{d (d_Y-1)}, 0, \ldots, 0} $.
Hence, the set $\{W_{\vec{\theta}}\}$ 
equals the set of transition matrices on ${\cal X}$ with full support.
\hfill$\Box$\end{proof}

\subsection{Local equivalence}
Next, we address the local equivalence problem at
a given independent-type ${\cal Y}$-indexed transition matrix $(W,V)$.
This is because 
we cannot necessarily distinguish all the elements of 
the above exponential family because due to the local equivalence problem.
Based on $(W,V)$,
we define the subspaces as
\begin{align}
(W,V)^{*}{\cal L}_{2,W,V}^I:=&
{\cal L}_{2,(W,V)}\cap (W,V)^{*} {\cal L}_{1,W,V}^I 
\\
(W,V)^{*}{\cal L}_{P,W,V}^I:=&
{\cal L}_{P,(W,V)}\cap (W,V)^{*}{\cal L}_{1,W,V}^I 
\\
(W,V)^{*}{\cal L}_{2,P,W,V}^I:=&
{\cal L}_{2,P,(W,V)}\cap (W,V)^{*}{\cal L}_{1,W,V}^I .
\end{align}
Then, we define ${\cal N}_2^I(({\cal X}^{2}\cup {\cal Y}\times {\cal X})_{(W,V)}) :=
(W,V)_*^{-1}({\cal L}_{2,\vec{W}}^I )$,
${\cal N}_P^I(({\cal X}^{2}\cup {\cal Y}\times {\cal X})_{(W,V)}) :=
(W,V)_*^{-1}({\cal L}_{P,\vec{W}}^I )$,
${\cal N}_{2,P}^I(({\cal X}^{2}\cup {\cal Y}\times {\cal X})_{(W,V)}) :=
(W,V)_*^{-1}({\cal L}_{2,P,\vec{W}}^I )$.
By using ${\cal N}^I(({\cal X}^{2}\cup {\cal Y}\times {\cal X})_{(W,V)})$ and these spaces,
Theorems \ref{L27-2} and \ref{27-3} characterize 
generators of the following condition;
the derivative of the direction of the generator vanishes in the observed distribution.
That is, the infinitesimal changes of the direction of the generator cannot be observed.


Then, ${\cal L}_{P,W,V}^I$ is written as follows.
\begin{align}
{\cal L}_{P,W,V}^I=&
\Big\{ 
(B,C)\in {\cal G}(({\cal X}^{2}\cup {\cal Y}\times {\cal X})_{(W,V)})
 \Big|
\hbox{Conditions \eqref{C1} and \eqref{C2} hold.}
\Big \},
\end{align}
where
Conditions \eqref{C1} and \eqref{C2} are defined as
\begin{align}
& B^T | u_{{\cal X}}\rangle = 0,\Label{C1} \\
& (W D(C_{y})+B D(V_y)) ({\cal V}^{k_{(P,(W,V))}}(P) +\Ker P^{k_{(W,V)}}[(W,V)] )
\subset  \Ker  P^{k_{(W,V)}}[(W,V)] 
\Label{C2}.
\end{align}
Here, ${\cal V}^{k_{(P,(W,V))}}(P) +\Ker P^{k_{(W,V)}}[(W,V)]$ expresses the subspace of 
${\cal V}_{{\cal X}}$ generated by $\Ker P^{k_{(W,V)}}[(W,V)]$ and the representatives of 
${\cal V}^{k_{(P,(W,V))}}(P)$ while
${\cal V}^{k_{(P,(W,V))}}(P)$ is a subspace of 
the quotient space ${\cal V}_{{\cal X}}/\Ker P^{k_{(W,V)}}[(W,V)]$.

To characterize other spaces ${\cal L}_{2,W,V}^I$ and ${\cal L}_{2,P,W,V}^I$,
for an element $x \in {\cal X}$, we define the subset $S(V)_x \subset{\cal X}$
by $S(V)_x := \{ x'\in {\cal X} |V_{*,x}=V_{*,x'} \}$.
For a subset $S \subset {\cal X}$,
we define the subspace ${\cal V}_{S}\subset {\cal V}_{{\cal X}}$ as 
the set of functions whose support is included in $S$.
The projection to ${\cal V}_{S}$ is denoted by $I_{S}$.
Then, the spaces ${\cal L}_{2,W,V}^I$ and ${\cal L}_{2,P,W,V}^I$
are characterized in the following theorem.

\begin{theorem}\Label{FET2}
For an independent-type ${\cal Y}$-indexed transition matrix $(W,V)$,
we have the following relations as subspaces of ${\cal G}(({\cal X}^{2}\cup {\cal Y}\times {\cal X})_{(W,V)})$;
\begin{align}
{\cal L}_{2,W,V}^I
=&
\{ ([W,A] , C ) | 
A^T| u_{{\cal X}} \rangle=0,~
W[D(V_y),A]=W D(C_y)
\} \Label{YF1T}\\
{\cal L}_{2,P,W,V}^I
=&
\{ ([W,A] , C ) | 
A^T| u_{{\cal X}} \rangle=0,~
A| P \rangle=0,~
W[D(V_y),A]=W D(C_y)
\} \Label{YF2T}.
\end{align}
\end{theorem}

\begin{proof}
For an element $(B,C) \in  {\cal L}_{1,(W,V)}^I$, 
$(B,C) \in {\cal L}_{2,(W,V)}^I$ if and only if  
there exists $A$ such that 
\begin{align} 
A^T |u_{{\cal X}}\rangle &=0 \Label{2-76} \\
WD(V_y) A-A WD(V_y) &= BD(V_y) +W D(C_y)\Label{2-9-10}.
\end{align}
Since $\sum_y D(V_y)=I$ and $\sum_y C_y=0$, 
taking the sum of \eqref{2-9-10} with respect to $y$, we have
\begin{align} 
[W,A]
=& W (\sum_y D(V_y)) A-A W (\sum_y D(V_y))
=\sum_y WD(V_y) A-A WD(V_y) \nonumber \\
=& \sum_y BD(V_y) +W D(C_y)
= B(\sum_y D(V_y)) +W D(\sum_y C_y)
= B.\Label{KKT}
\end{align}
Combining \eqref{2-9-10} and \eqref{KKT}, we have
\begin{align}
W [D(V_y), A]=W D(C_y).\Label{2-9-11T} 
\end{align}
Conversely, \eqref{KKT} and \eqref{2-9-11T} imply \eqref{2-9-10} and 
$\sum_{y \in {\cal Y}} C_y=0$.
Further,  \eqref{2-76} and \eqref{KKT} imply $B^T |u_{{\cal X}}\rangle=0$.
Hence, we obtain the relation \eqref{YF1T}.

Since ${\cal L}_{2,P,(W,V)}$ is given from ${\cal L}_{2,(W,V)}$ by adding the condition $
A| P \rangle=0$,
we obtain the relations \eqref{YF2T}.
\hfill$\Box$\end{proof}

Based on Theorem \ref{FET2}, we can characterize the subspaces 
${\cal L}_{2,W,V}^I$ and ${\cal L}_{2,P,W,V}^I$
as the following two corollaries.

\begin{corollary}\Label{FET}
For an independent-type ${\cal Y}$-indexed transition matrix $(W,V)$,
we assume that 
$W$ is invertible. 
Then, we have the following relations as subspaces of ${\cal G}(({\cal X}^{2}\cup {\cal Y}\times {\cal X})_{(W,V)})$;
\begin{align}
{\cal L}_{2,W,V}^I
=&
\{ ([W,A] , 0 ) | 
A^T| u_{{\cal X}} \rangle=0,~
[I_{S(V)_x},A]=0
\} \Label{YF1}\\
{\cal L}_{2,P,W,V}^I
=&
\{ ([W,A] , 0 ) | 
A^T| u_{{\cal X}} \rangle=0,~
A| P \rangle=0,~
[I_{S(V)_x},A]=0
\} \Label{YF2}.
\end{align}
\end{corollary}

\begin{proof}
We choose $A$ and $C$ such that
\begin{align} 
A^T |u_{{\cal X}}\rangle &=0  \\
W[D(V_y),A]&=W D(C_y)
\Label{2-9-10V}.
\end{align}
Since $W$ is invertible, 
$[D(V_y),A]= D(C_y)$.
Since the diagonal elements of $[D(V_y),A]$ are zero, 
we have 
\begin{align}
D(V_y) A &=AD(V_y),\Label{KFC}\\
C_y&=0.\Label{KFCe}
\end{align}
When the $(x,x')$ component of $A$ is not zero for $x \neq x'$,
we have $V(y|x)=V(y|x')$, which implies the relation $x' \in S(V)_x$. Hence, we have
\begin{align}
[I_{S(V)_x},A]=0 . \Label{KFD}
\end{align}
Conversely, the combination of \eqref{KFCe} and \eqref{KFD} implies \eqref{2-9-10V}.
Hence, we obtain the relation \eqref{YF1}.
Since ${\cal L}_{2,P,(W,V)}$ is given from ${\cal L}_{2,(W,V)}$ by adding the condition $
A| P \rangle=0$,
we obtain the relations \eqref{YF2}.
\hfill$\Box$\end{proof}

\begin{corollary}\Label{FETE}
For an independent-type ${\cal Y}$-indexed transition matrix $(W,V)$,
we assume that $W(x|x')=\frac{1}{d}$.
Then, we have the following relations as subspaces of ${\cal G}(({\cal X}^{2}\cup {\cal Y}\times {\cal X})_{(W,V)})$;
\begin{align}
{\cal L}_{2,W,V}^I
=&
\{ ([W,A] , A^T V_y ) | 
A^T| u_{{\cal X}} \rangle=0
\} \Label{YF3}\\
{\cal L}_{2,P,W,V}^I
=&
\{ ([W,A] , A^T V_y ) | 
A^T| u_{{\cal X}} \rangle=0,~
A| P \rangle=0
\} \Label{YF4}.
\end{align}
\end{corollary}

\begin{proof}
The condition $W[D(V_y),A]=W D(C_y)$ is equivalent to 
the condition $([D(V_y),A])^T |u_{{\cal X}}\rangle= D(C_y)|u_{{\cal X}}\rangle$.
Since $A^T| u_{{\cal X}} \rangle=0 $, this condition is equivalent to 
$ A^T V_y=C_y$.
Hence, the desired statement.
\end{proof}

Using Corollary \ref{FET}, we have the following corollary.
\begin{corollary}\Label{VDT}
We assume that all the vectors $V_{*,x}$ are different
and the relations $ \Ker P^{k_{(W,V)}}[(W,V)]=\{0\}$ and ${\cal V}^{k_{(P,\vec{W})}}(P)
={\cal V}_{{\cal X}}$.
Then,
$ {\cal L}_{2,W,V}^I=  {\cal L}_{2,P,W,V}^I={\cal L}_{P,W,V}^I=\{0\}$.
\end{corollary}

\begin{proof}
Since all the vectors $V_{*,x}$ are different,
we have $S(V)_x=\{x\}$.
In this case, the condition
$A^T| u_{{\cal X}} \rangle=0,~[I_{S(V)_x},A]=0$
implies that $A=0$.
Due to Corollary \ref{FET},
this assumption guarantees that 
${\cal L}_{2,W,V}^I$
and ${\cal L}_{2,P,W,V}^I$
are $\{0\}$, i.e.,
$ \dim {\cal L}_{2,W,V}^I= \dim {\cal L}_{2,P,W,V}^I=0$.
The relations $ \Ker P^{k_{(W,V)}}[(W,V)]=\{0\}$ and ${\cal V}^{k_{(P,\vec{W})}}(P)
={\cal V}_{{\cal X}}$ imply the relation 
$\dim {\cal L}_{P,W,V}^I=0$.
\hfill$\Box$\end{proof}

Theorem \ref{VDT} means that 
the indistinguishable subspaces 
$ {\cal L}_{2,W,V}^I,{\cal L}_{2,P,W,V}^I$, and ${\cal L}_{P,W,V}^I$
vanish in a usual case.
More precisely, since all the vectors $V_{*,x}$ are different almost everywhere with respect to $V$,
Lemma \ref{2-15-B} and Theorem \ref{VDT} guarantee 
the relation $ \dim {\cal L}_{2,W,V}^I= \dim {\cal L}_{2,P,W,V}^I=\dim {\cal L}_{P,W,V}^I=0$
almost everywhere with respect to $W,V$.
Hence, Theorem \ref{VDT} guarantees that 
Example \ref{ERT} has dimension $ d^2\cdot (d_Y-1) $ except for a measure-zero set. 
Remember that, as shown in Lemma \ref{LKD}, 
Example \ref{ERT} characterizes all of the independent-type ${\cal Y}$-indexed transition matrices with full support.
Hence, we call an element of such a measure-zero set 
{\it an independent-type singular point}. 
Since the full model without the independent-type condition 
has the dimension $ d^2\cdot (d_Y-1) $ at non-singular points,
we have the inequality $d(d+d_Y-2)\le d^2\cdot (d_Y-1)$,
and the equality holds only when $d_Y=2 $.
Hence, we can expect that the condition of Lemma \ref{HGR} holds with non-zero measure with respect to $V,W$
when $d_Y=2 $. 

\begin{remark}\Label{R2}
Here, we compare the notations of
${\cal Y}$-indexed transition matrix $\vec{W}$ (Fig. \ref{3model}), 
independent-type ${\cal Y}$-indexed transition matrix $(W,V)$ (Fig. \ref{2model}), 
and
independent-type ${\cal Y}$-indexed transition matrix $(W,f)$ (Fig. \ref{1model}),
where the third case express the transition matrix $V$ is given as a deterministic function $f$.
Although $P^k[\vec{W}](y_k, \ldots, y_1|x')$ is given as \eqref{LEV} in the first case,
it is written as \eqref{LEV2} in the second case. 
In the third case, it is described as
$\sum_{x\in {\cal X}} I_{f^{-1}(y_k)} W \cdots W I_{f^{-1}(y_1)}(x|x')$.
Clearly, the description \eqref{LEV} of the first case is shortest.

Indeed, due to this simplicity, 
we can easily find that an exponential family of
${\cal Y}$-indexed transition matrices
is a special case of 
an exponential family of transition matrices.
It is not so easy to find that 
an exponential family of
independent-type ${\cal Y}$-indexed transition matrices
is a special case of 
an exponential family of transition matrices
without considering the relation between 
independent-type ${\cal Y}$-indexed transition matrix
and ${\cal Y}$-indexed transition matrix.
These comparisons express the merit of notation of  
${\cal Y}$-indexed transition matrix $\vec{W}$ (Fig. \ref{3model}).
\end{remark}


\section{Two-hidden-state case in conditionally independent model}\Label{s-Fi}
We consider the case with $d=2$.
In this case, since
the subspace ${\cal N}^I(({\cal X}^{2}\cup {\cal Y}\times {\cal X})_{(W,V)})$
equals the subspace ${\cal N}(({\cal Y}\times {\cal X}^{2})_{\vec{W}}) $, 
it is given by \eqref{eq2-13-9B}.

\subsection{Non-singular points}
First, we assume that $W$ is invertible and the relation
\begin{align}
V(y|0)=V(y|1) \hbox{ for }\forall y \in {\cal Y}\Label{SJT}
\end{align}
does not hold.
In this case,
$\Ker P^{k_{(W,V)}}[(W,V)] =\{0\}$ and ${\cal V}^{k_{(P,(W,V))}}(P) = {\cal V}_{{\cal X}}$.
Hence,
${\cal L}_{P,W,V}^I=\{0\}$.
Corollary \ref{VDT} guarantees that 
$ {\cal L}_{2,W,V}^I=  {\cal L}_{2,P,W,V}^I={\cal L}_{P,W,V}^I=\{0\}$.
Hence, the above condition guarantees that the point is non-singular.
Then, $l'=d_Y-1$ and $l=2d_Y$.
The generators are chosen as in Example \ref{ERT}.

\subsection{Singular points}
The subset of singular elements equals the set of non-memory cases, which has two cases.
As the three cases, we assume that 
the relation \eqref{SJT} does not hold and
$W$ is not invertible, i.e., $W=\left(
\begin{array}{cc}
\frac{1}{2} & \frac{1}{2} \\
\frac{1}{2} & \frac{1}{2} 
\end{array}
\right)$.
In this case,
$\Ker P^{k_{(W,V)}}[(W,V)] =\{0\}$ and ${\cal V}^{k_{(P,(W,V))}}(P) = {\cal V}_{{\cal X}}$.
Hence,
${\cal L}_{P,W,V}^I=\{0\}$.
The dimensions of
$ {\cal L}_{2,W,V}^I$ and $  {\cal L}_{2,P,W,V}^I$
are given by the dimensions of $A$ satisfying the constraint given in Corollary \ref{FETE}.
Hence, $\dim {\cal L}_{2,W,V}^I=2$.
When $P$ is the stationary distribution, i.e., the uniform distribution, $ \dim {\cal L}_{2,P,W,V}^I=2$.
Otherwise, it  is $1$.
The dimension of the quotient space 
${\cal G}_1(({\cal X}^{2}\cup {\cal Y}\times {\cal X})_{(W,V)})
/({\cal N}_2^I(({\cal X}^{2}\cup {\cal Y}\times {\cal X})_{(W,V)})+
{\cal N}_{P_W}^I(({\cal X}^{2}\cup {\cal Y}\times {\cal X})_{(W,V)}))$
is $2d_Y-2$.
In this case, the initial $l'=d_Y-1$ generators given in \eqref{FR1} and \eqref{FR2} of Example \ref{ERT} express the variation inside of the set of
singular points.
The remaining $d_Y-1$ generators express the difference from this set of singular point.
When $P$ is the  stationary distribution,
the quotient space 
${\cal G}_1(({\cal X}^{2}\cup {\cal Y}\times {\cal X})_{(W,V)})
/({\cal N}_{2,P}^I(({\cal X}^{2}\cup {\cal Y}\times {\cal X})_{(W,V)})+
{\cal N}_{P}^I(({\cal X}^{2}\cup {\cal Y}\times {\cal X})_{(W,V)}))$
has the same structure as the above case.
When $P$ is not the  stationary distribution,
the quotient space 
${\cal G}_1(({\cal X}^{2}\cup {\cal Y}\times {\cal X})_{(W,V)})
/({\cal N}_{2,P}^I(({\cal X}^{2}\cup {\cal Y}\times {\cal X})_{(W,V)})+
{\cal N}_{P}^I(({\cal X}^{2}\cup {\cal Y}\times {\cal X})_{(W,V)}))$
has dimension $2d_Y-1$.
The same initial $d_Y-1$ generators express the variation inside of the set of
singular points, and the remaining $d_Y$ generators express the difference from this set of singular point.

As the second case, we consider the case with the relation \eqref{SJT} with invertible $W$.
Choosing $v_1:= e_1+e_2$ and $v_2:= e_1-e_2$,
we have 
$\Ker P^{k_{(W,V)}}[(W,V)] =<v_2>$ and ${\cal V}^{k_{(P,(W,V))}}(P) = {\cal V}_{{\cal X}}$.
Hence, $(B,C)\in  {\cal L}_{P,W,V}^I$ if and only if
the matrix $WD(C_y)+B D(V_y)$ has the form
$\left(
\begin{array}{ll}
a & b \\
-a & -b
\end{array}
\right)$.
Also, the matrix $B$ has the form
$\left(
\begin{array}{ll}
b_1 & b_2 \\
-b_1 & -b_2
\end{array}
\right)$.
$D(V_y)$ is a scalar times of the identity matrix. 
When $C_y=
\left(
\begin{array}{ll}
c\\
-c
\end{array}
\right)$,
the matrix $WD(C_y)$ has the form
$\left(
\begin{array}{ll}
c W(0|0) & -c W(0|1) \\
c W(1|0) & -c W(1|1) 
\end{array}
\right)$.
Since $W(x|x')\ge 0$, 
the real number $c$ needs to be $0$.
Hence, we find that
\begin{align}
 {\cal L}_{P,W,V}^I
=\left\{
\left(
\left(
\begin{array}{ll}
b_1 & b_2 \\
-b_1 & -b_2
\end{array}
\right), 0
\right)
\right\}.
 \end{align}

Hence, $ {\cal L}_{P,W,V}^I=\{(B,0)| B^T| u_{{\cal X}} \rangle=0 \}$
because the space $\{(B,0)| B^T| u_{{\cal X}} \rangle=0 \}$ has dimension $2$.
Due to the RHSs of \eqref{YF1} and \eqref{YF2},
we have ${\cal L}_{2,W,V}^I,{\cal L}_{2,P,W,V}^I
\subset \{(B,0)| B^T| u_{{\cal X}} \rangle=0 \}={\cal L}_{P,W,V}^I$.
The dimension of the quotient space 
${\cal G}_1(({\cal X}^{2}\cup {\cal Y}\times {\cal X})_{(W,V)})
/({\cal N}_2^I(({\cal X}^{2}\cup {\cal Y}\times {\cal X})_{(W,V)})+
{\cal N}_{P_W}^I(({\cal X}^{2}\cup {\cal Y}\times {\cal X})_{(W,V)}))$
is $2d_Y-2$.
The initial $d_Y-1$ generators given in \eqref{FR1} and \eqref{FR2} of Example \ref{ERT} express the variation inside of the set of
singular points.
The remaining $d_Y-1$ generators in \eqref{FR3} and \eqref{FR4} of Example \ref{ERT}
express the difference from this set of singular point.
The quotient space 
${\cal G}_1(({\cal X}^{2}\cup {\cal Y}\times {\cal X})_{(W,V)})
/({\cal N}_{2,P}^I(({\cal X}^{2}\cup {\cal Y}\times {\cal X})_{(W,V)})+
{\cal N}_{P}^I(({\cal X}^{2}\cup {\cal Y}\times {\cal X})_{(W,V)}))$
has the same structure as the above case, regardless of $P$.

As the third case, we consider the case with the relation \eqref{SJT} and
$W=\left(
\begin{array}{cc}
\frac{1}{2} & \frac{1}{2} \\
\frac{1}{2} & \frac{1}{2} 
\end{array}
\right)$.
Then, we have 
$\Ker P^{k_{(W,V)}}[(W,V)] =<v_2>$ and ${\cal V}^{k_{(P,(W,V))}}(P) = <v_1>$.
Hence, 
$(B,C)\in  {\cal L}_{P,W,V}^I$ if and only if
the matrix $WD(C_y)+B D(V_y)$ has the form
$\left(
\begin{array}{cc}
a & b \\
-a & -b
\end{array}
\right)$.
Hence, $ {\cal L}_{P,W,V}^I=\{(B,0)| B^T| u_{{\cal X}} \rangle=0 \}$.
In this case, $V_y$ is a scalar times of 
$
\left(
\begin{array}{c}
1\\
1
\end{array}
\right)
$. Hence, Corollary \ref{FETE} guarantees that
${\cal L}_{2,W,V}^I,{\cal L}_{2,P,W,V}^I
\subset \{(B,0)| B^T| u_{{\cal X}} \rangle=0 \}={\cal L}_{P,W,V}^I$.
Hence, the remaining discussion is the same as the second case.

\subsection{Parametrization}
In this subsection, we employ the parametrization given in Example \ref{ERT}.
\subsubsection{Case with $d_Y=2$}
When $d_Y=2$, we choose $W$ and $V$ as
\begin{align}
W=V=
\left(
\begin{array}{cc}
\frac{1}{2} & \frac{1}{2} \\
\frac{1}{2} & \frac{1}{2} 
\end{array}
\right).
\end{align}
The subset of singular elements equals the set of non-memory cases, 
which can be characterized as $\theta_2=0$ or $\theta_3=\theta_4=0$.
In the former case, the parameters $\theta_3$ and $\theta_4$ are redundant parameters,
and 
in the latter case, the parameter $\theta_2$ is a redundant parameter.
Both cases are equivalent to the binomial distribution. 
Hence, the set of non-singular elements are given as
$\bR \times (\bR\setminus \{0\})\times (\bR^2\setminus \{(0,0)\})$,
which can be divided into two connected components
$\bR \times (0,\infty)\times (\bR^2\setminus \{(0,0)\})$ 
and
$\bR \times (-\infty,0)\times (\bR^2\setminus \{(0,0)\})$.
Each connected component has a one-to-one correspondence to 
non-singular elements divided by the equivalence class.

\subsubsection{Case with $d_Y\ge 3$}
When $d_Y\ge 3$, we choose $W$ and $V$ as
\begin{align}
W=
\left(
\begin{array}{cc}
\frac{1}{2} & \frac{1}{2} \\
\frac{1}{2} & \frac{1}{2} 
\end{array}
\right),\quad
V=
\left(
\begin{array}{cc}
\frac{1}{d_Y} & \frac{1}{d_Y} \\
\frac{1}{d_Y} & \frac{1}{d_Y} \\
\vdots & \vdots \\
\frac{1}{d_Y} & \frac{1}{d_Y} 
\end{array}
\right).
\end{align}
The subset of singular elements equals the set of non-memory cases, 
which can be characterized as $\theta_{d_Y}=\cdots =\theta_{2d_y-2}=0$ or 
$\theta_{2d_Y-1}=\theta_{2d_Y}=0$.
In the former case, the parameters $\theta_{2d_Y-1}$ and $\theta_{2d_Y}$ are redundant parameters,
and 
in the latter case, the parameters $\theta_{d_Y},\cdots ,\theta_{2d_y-2}$ are redundant parameters.
We denote the set 
$(\bR^{d_Y-1} \times \{(0,\ldots,0)\}\times \bR^{2})
\cup (\bR^{2d_Y-2} \times \{(0,0)\})$ by $\Theta_{NM}$.
That is, the set $\bR^{2 d_Y} \setminus \Theta_{NM}
=\bR^{d_Y-1}\times (\bR^{d_Y-1}\setminus \{(0,\ldots,0)\})
\times (\bR^{2}\setminus \{(0,0)\})$ 
equals to the set of non-singular elements.
However, it is impossible to divide the set $\bR^{2 d_Y} \setminus \Theta_{NM}$
into components satisfying the following conditions.
(1) Each component is an open set.
(2) Each component gives a one-to-one parametrization for non-singular elements.
This is because the set $\bR^{2 d_Y} \setminus \Theta_{NM}$ is connected.
Hence, we need to adopt duplicated parametrization when the parametric space is needed to be open.

\section{Conclusion}\Label{s9}
In Section \ref{s5}, we have introduced the concept of ${\cal Y}$-indexed transition matrix 
to describe a hidden Markov process, which is a more general formulation
than the conventional formulation for a hidden Markov process.
In fact, as explained in Remark \ref{R2}, this notion is more useful to describe the equivalence problem. 
Then, in Section \ref{s6}, we have formulated an exponential family of 
${\cal Y}$-indexed transition matrices as a special case of 
an exponential family of transition matrices.
In this definition, the generators are given as functions of hidden and observed states. 
Then, we have introduced an equivalence relation for generators, 
which is equivalent to the distinguishability of infinitesimal changes
based on the observed data (See Theorem \ref{27-3}).
In Section \ref{s6-3}, based on this equivalence relation, we have proposed 
a method to choose the parametrization of the transition matrix to describe the hidden Markov process. 
In this parametrization, we have shown that
only in a measure-zero point, the number of independent generators is smaller than other points.
We define singular points as such measure-zero points.

In addition, in Section \ref{s7}, we have clarified the structure of the tangent space of all points including singular points
when the number of hidden state is $2$.
Next, In Section \ref{s8}, we have applied obtained results to the conventional case, which is called independent-type and is characterized by a pair of transition matrices.
We have derived a necessary and sufficient condition for being independent-type.
Also, we have clarified the forms of an exponential family of 
${\cal Y}$-indexed transition matrices and 
the local equivalence condition in this case.
Based on this equivalence relation, we have proposed 
a method to choose the parametrization of the transition matrix 
for this case.

We have several open problems as follows.
First, while we have shown that the dimension of non-singular points in the independent-type case is the same 
as that in the case of general ${\cal Y}$-indexed transition matrices when 
the number of observed states is $2$,
we could not clarify whether there exists a general ${\cal Y}$-indexed transition matrix that cannot 
be reduced to an independent-type one in this case.
This is a future problem.
Another remaining problem is a characterization of the tangent space 
when the transition matrix $V$ is given as a deterministic function $f$.
In this case, 
due to Corollary \ref{FET}, the indistinguishable subspaces 
$ {\cal L}_{2,W,V}^I$ and $ {\cal L}_{2,P,W,V}^I$, and ${\cal L}_{P,W,V}^I$
do not vanish.
Hence, the structure of the tangent space is complicated.
The determination of the parametrization of this case with taking the local equivalence into account
is another future problem.

\section*{Acknowledgment}
The author is very grateful to 
Professor Takafumi Kanamori, Professor Vincent Y. F. Tan, and Dr. Wataru Kumagai
for helpful discussions and comments.
The works reported here were supported in part by 
the JSPS Grant-in-Aid for Scientific Research 
(B) No. 16KT0017, (A) No.17H01280, 
the Okawa Research Grant
and Kayamori Foundation of Informational Science Advancement.

\section*{Conflict of interest}
On behalf of all authors, the corresponding author states that there is no conflict of interest. 

\appendix

\section{Proof of Theorem \ref{L27-2}}\Label{A1}
It is enough to discuss the one-parameter case.
Since ${\bf (B2)} \Rightarrow {\bf (B3)}$ is trivial, we will show only 
${\bf (B1)} \Rightarrow {\bf (B2)}$ and ${\bf (B3)} \Rightarrow {\bf (B1)}$.

\PF{${\bf (B1)} \Rightarrow {\bf (B2)}$}
Assume {\bf (B1)}.
There exist $A \in {\cal M}({\cal V}_{\cal X})$ 
and $(B_{y})_{y\in {\cal Y}} \in {\cal L}_{2,\vec{W}}$
such that
$A|P\rangle =0$, $A^T|u_{{\cal X}}\rangle=0 $,
$B_y ({\cal V}^{k_{(P,\vec{W})}}(P) +\Ker P^{k_{\vec{W}}}[\vec{W}] )
\subset \Ker P^{k_{\vec{W}}}[\vec{W}]$,
$(\sum_y B_y)^T |u_{{\cal X}}\rangle=0 $,
and 
$\frac{d }{d \theta} W_{\theta,y}|_{{\theta}=0}=B_y+[W_y,A]$
for any $y \in {\cal Y}$.
Then,
\begin{align}
& \frac{d }{d \theta}
P^{k}[\vec{W}_{\vec{\theta}}]\cdot P (y_1, \ldots, y_k)|_{{\theta}=0}
=\frac{d }{d \theta}
(\langle u_{{\cal X}}| W_{\theta,y_k}W_{\theta,y_{k-1}} \ldots W_{\theta,y_1} |P\rangle)|_{{\theta}=0} \nonumber\\
=&
\langle u_{{\cal X}}| 
(\frac{d }{d \theta} W_{\theta,y_k}|_{{\theta}=0})
W_{y_{k-1}} \ldots W_{y_1} |P\rangle
+
\langle u_{{\cal X}}| 
W_{y_k}
(\frac{d }{d \theta} W_{\theta,y_{k-1}}|_{{\theta}=0}) 
\ldots W_{y_1} |P\rangle \nonumber \\
&+ \cdots +
\langle u_{{\cal X}}| 
W_{y_k}
W_{y_{k-1}} \ldots (\frac{d }{d \theta} W_{\theta,y_1}|_{{\theta}=0}) 
 |P\rangle \\
=&
\langle u_{{\cal X}}| 
B_{y_k}
W_{y_{k-1}} \ldots W_{y_1}  |P\rangle
+
\langle u_{{\cal X}}| 
W_{y_k}
B_{y_{k-1}}
\ldots W_{y_1}  |P\rangle\nonumber \\
&+ \cdots +
\langle u_{{\cal X}}| 
W_{y_k}
W_{y_{k-1}} \ldots 
B_{y_1} |P\rangle \nonumber\\
&+
\langle u_{{\cal X}}| 
[W_{y_k},A]
W_{y_{k-1}} \ldots W_{y_1}  |P\rangle
+
\langle u_{{\cal X}}| 
W_{y_k}
[W_{y_{k-1}},A]
\ldots W_{y_1}  |P\rangle
\nonumber \\
&+ \cdots +
\langle u_{{\cal X}}| 
W_{y_k}
W_{y_{k-1}} \ldots 
[W_{y_1},A]
 |P\rangle \\
\stackrel{(a)}{=}&
-\langle u_{{\cal X}}| 
A W_{y_k} W_{y_{k-1}} \ldots W_{y_1}  |P\rangle
+\langle u_{{\cal X}}| W_{y_k} W_{y_{k-1}} \ldots 
W_{y_1}A |P\rangle 
\stackrel{(b)}{=}
0,\Label{9-27-3}
\end{align}
where $(a)$ follows from the fact that the image of $B_y$ is included in 
$\Ker P^{k_{\vec{W}}}[\vec{W}]$,
and $(b)$ follows from the properties of $A$.
So, we obtain {\bf (B2)}.

\PF{${\bf (B3)} \Rightarrow {\bf (B1)}$}
Assume {\bf (B3)}.
We define 
$\vec{W}_{\theta,y}':= 
\vec{W}_{y}+ \theta  
\frac{d}{d\theta}\vec{W}_{\theta,y}|_{\theta=0}$.
So, we have
$P^{k_{\vec{W}}+k_{(P,\vec{W})}+1}[\vec{W}_{{\theta}}']\cdot P 
=P^{k_{\vec{W}}+k_{(P,\vec{W})}+1}[\vec{W}]\cdot P $
and
\begin{align}
\lim_{\theta \to 0} \frac{\vec{W}_{\theta,y}-\vec{W}_{\theta',y}}{\theta}=0
\Label{9-27-1}.
\end{align}
Theorem \ref{L-9-27} guarantees that the pair of $\vec{W}$ and $P$ is equivalent to 
the pair of $\vec{W}_\theta'$ and $P$.
Thus, Theorem \ref{L-9-27} guarantees that
there exist an invertible map $T_\theta$ on ${\cal V}_{{\cal X}}$
and an element $(B_{\theta,y})_{y\in {\cal Y}} \in {\cal L}_{2,\vec{W}}$
such that $T_\theta P=P$,
$B_{\theta,y} ({\cal V}^{k_{(P,\vec{W})}}(P) +\Ker P^{k_{\vec{W}}}[\vec{W}] )
\subset \Ker P^{k_{\vec{W}}}[\vec{W}]$ and 
$\vec{W}_{\theta ,y}'= T_\theta^{-1}(W_y+B_{\theta,y})T_\theta$.

Now, taking the derivative at $\theta=0$, we have
$\frac{d}{d\theta}\vec{W}_{\theta ,y}'|_{\theta=0}=[W_y,A]+B_y$,
where 
$A:=\frac{d}{d\theta}T_\theta|_{\theta=0}$
and $B_y:=\frac{d}{d\theta}B_{\theta,y}|_{\theta=0}$.
The condition $T_\theta P=P$ implies that
\begin{align}
A|P\rangle=0.\Label{27-1}
\end{align}
Using \eqref{9-27-1}, we have
\begin{align}
\frac{d}{d\theta}\vec{W}_{\theta ,y}|_{\theta=0}
=[W_y,A]+B_y
\Label{9-27-2}.
\end{align}

Since the relation $(\vec{W}_{\theta ,y})^T|u_{{\cal X}}\rangle=|u_{{\cal X}}\rangle$
implies $(\sum_y \frac{d}{d\theta}\vec{W}_{\theta ,y}|_{\theta=0})^T|u_{{\cal X}}\rangle=0$,
$(\sum_y [W_y,A]+B_y)^T |u_{{\cal X}}\rangle=0$.
Since $B_y^T |u_{{\cal X}}\rangle=0$, 
we have $ ([\sum_y W_y,A])^T |u_{{\cal X}}\rangle=0$.
So, $(\sum_y W_y)^T A^T |u_{{\cal X}}\rangle
= A^T(\sum_y W_y)^T  |u_{{\cal X}}\rangle= A^T  |u_{{\cal X}}\rangle$.
That is, $A^T  |u_{{\cal X}}\rangle$ is an eigenvector of $(\sum_y W_y)^T$ 
with eigenvalue $1$.
So, $A^T  |u_{{\cal X}}\rangle$ is written as $c  |u_{{\cal X}}\rangle$ with a constant $c$, i.e.,
\begin{align}
A^T  |u_{{\cal X}}\rangle=c  |u_{{\cal X}}\rangle
\Label{27-2}.
\end{align}
Now, we calculate 
$ \frac{d }{d \theta}
P^{k}[\vec{W}_{\vec{\theta}}]\cdot P (y_1, \ldots, y_k)|_{{\theta}=0}$
by using the same discussion as \eqref{9-27-3}.
So, we have
\begin{align}
& \frac{d }{d \theta}
P^{k}[\vec{W}_{\vec{\theta}}]\cdot P (y_1, \ldots, y_k)|_{{\theta}=0}
\nonumber \\
= &
-\langle u_{{\cal X}}| 
A W_{y_k} W_{y_{k-1}} \ldots W_{y_1}  |P\rangle
+
\langle u_{{\cal X}}| 
W_{y_k}
W_{y_{k-1}} \ldots 
W_{y_1}A |P\rangle 
\nonumber \\
\stackrel{(a)}{=} &
-c 
\langle u_{{\cal X}}| 
 W_{y_k} W_{y_{k-1}} \ldots W_{y_1}  |P\rangle,
\end{align}
where $(a)$ follows from \eqref{27-1} and \eqref{27-2}.
Since 
$\langle u_{{\cal X}}| 
 W_{y_k} W_{y_{k-1}} \ldots W_{y_1}  |P\rangle >0$ 
and the LHS is zero, we have $c=0$.
Thus, we obtain {\bf (B1)}.

\section{Proof of Theorem \ref{27-3}}\Label{A2}
It is enough to discuss the one-parameter case.
Since ${\bf (C2)} \Rightarrow {\bf (C3)}$ is trivial, we will show only 
${\bf (C1)} \Rightarrow {\bf (C2)}$ and ${\bf (C3)} \Rightarrow {\bf (C1)}$.

\PF{${\bf (C1)} \Rightarrow {\bf (C2)}$}
Assume {\bf (C1)}.
There exist a real number $c\in \mathbb{R}$,
$A \in {\cal M}({\cal V}_{\cal X})$, 
and $(B_{y})_{y\in {\cal Y}} \in {\cal L}_{2,\vec{W}}$
such that
\begin{align}
A^T|u_{{\cal X}}\rangle & =c |u_{{\cal X}}\rangle ,\Label{27-a} \\
B_y ({\cal V}^{k_{(P,\vec{W})}}(P) +\Ker P^{k_{\vec{W}}}[\vec{W}] )
&\subset \Ker P^{k_{\vec{W}}}[\vec{W}], \Label{27-b} \\
(\sum_y B_y)^T |u_{{\cal X}}\rangle &=0 \Label{27-c} \\
\frac{d }{d \theta} W_{\theta,y}|_{{\theta}=0} &=B_y+[W_y,A]
\Label{27-d} 
\end{align}
for any $y \in {\cal Y}$.
Define the vector $Q:= \frac{d}{d\theta} P_{\vec{W}_{\theta}} |_{{\theta}=0}$.
Since 
\begin{align}
(\sum_{y \in {\cal Y}} W_{\theta,y}) | P_{\vec{W}_{\theta}} \rangle
=| P_{\vec{W}_{\theta}} \rangle
\Label{28-1},
\end{align}
we have
\begin{align}
&(\sum_{y \in {\cal Y}} B_y) | P_{\vec{W}} \rangle
+(\sum_{y \in {\cal Y}} W_y)A | P_{\vec{W}} \rangle
-A | P_{\vec{W}} \rangle
+ (\sum_{y \in {\cal Y}} W_{y}) | Q \rangle
\nonumber \\
\stackrel{(a)}{=} &
\Big( \Big(\sum_{y \in {\cal Y}} B_y\Big)
+\Big( \Big(\sum_{y \in {\cal Y}} W_y\Big)A 
-A\Big(\sum_{y \in {\cal Y}} W_y\Big)\Big)
\Big) | P_{\vec{W}} \rangle
+ \Big(\sum_{y \in {\cal Y}} W_{y}\Big) | Q \rangle
\nonumber \\
=&
\Big(\sum_{y \in {\cal Y}} B_y+[W_y,A]\Big) | P_{\vec{W}} \rangle
+ \Big(\sum_{y \in {\cal Y}} W_{y}\Big) | Q \rangle
\nonumber \\
\stackrel{(b)}{=}&| Q \rangle,
\end{align}
where $(a)$ and $(b)$ follow from \eqref{28-1} and its derivative, respectively.

That is,
\begin{align}
\Big(\sum_{y \in {\cal Y}} W_y\Big)
\Big(| Q \rangle + A |P_{\vec{W}} \rangle\Big)
=
\Big(| Q \rangle + A |P_{\vec{W}} \rangle\Big)
-\Big(\sum_{y \in {\cal Y}} B_y\Big) | P_{\vec{W}} \rangle.
\end{align}
Since 
$\Big(\sum_{y \in {\cal Y}} W_y\Big) \Big(\sum_{y \in {\cal Y}} B_y\Big) | P_{\vec{W}} \rangle
=0$,
we have
\begin{align}
\bigg(\sum_{y \in {\cal Y}} W_y\bigg)
\Bigg(| Q \rangle + A |P_{\vec{W}} \rangle
-\bigg(\sum_{y \in {\cal Y}} B_y\bigg) | P_{\vec{W}} \rangle \Bigg)
=
\Bigg(| Q \rangle + A |P_{\vec{W}} \rangle
-\bigg(\sum_{y \in {\cal Y}} B_y\bigg) | P_{\vec{W}} \rangle \Bigg).
\end{align}
Due to the uniqueness of the eigenvector of 
$\bigg(\sum_{y \in {\cal Y}} W_y\bigg)$ with eigenvalue $1$,
we have
\begin{align}
| Q \rangle + A |P_{\vec{W}} \rangle
-\bigg(\sum_{y \in {\cal Y}} B_y\bigg) | P_{\vec{W}} \rangle 
= c' |P_{\vec{W}} \rangle \Label{27-e} 
\end{align}
with a constant $c' \in \mathbb{R}$.

Since 
\begin{align}
\sum_{y \in {\cal Y}} \langle u_{{\cal X}}| W_{\theta,y}|P_{\vec{W}_{\theta}}
\rangle
=1,
\Label{27-g}
\end{align}
we have
\begin{align}
& c '-c
\nonumber \\
=& 
c '\langle u_{{\cal X}}| \bigg(\sum_{y \in {\cal Y}} W_{y}\bigg)  |P_{\vec{W}} \rangle
-c\langle u_{{\cal X}} |P_{\vec{W}} \rangle
\nonumber \\
\stackrel{(a)}{=}& 
c '\langle u_{{\cal X}}| \bigg(\sum_{y \in {\cal Y}} W_{y}\bigg)  |P_{\vec{W}} \rangle
-\langle u_{{\cal X}}| A |P_{\vec{W}} \rangle
\nonumber \\
=& \langle u_{{\cal X}}| \bigg(\sum_{y \in {\cal Y}} W_{y}\bigg) 
\Bigg( c' |P_{\vec{W}} \rangle
-A |P_{\vec{W}} \rangle
+\bigg(\sum_{y \in {\cal Y}} B_y\bigg) | P_{\vec{W}} \rangle \Bigg)
\nonumber \\
\stackrel{(b)}{=}& 
\langle u_{{\cal X}}| \bigg(\sum_{y \in {\cal Y}} W_{y}\bigg) |Q\rangle
\nonumber \\
=& \langle u_{{\cal X}}| A|P_{\vec{W}} \rangle
-\langle u_{{\cal X}}| A|P_{\vec{W}} \rangle
+
\langle u_{{\cal X}}| \bigg(\sum_{y \in {\cal Y}} W_{y}\bigg) |Q\rangle
\nonumber \\
=&  \langle u_{{\cal X}}| \Bigg(\bigg(\sum_{y \in {\cal Y}} W_y\bigg)A -
A \bigg(\sum_{y \in {\cal Y}} W_y\bigg) \Bigg) |P_{\vec{W}} \rangle
+
\langle u_{{\cal X}}| \sum_{y \in {\cal Y}} W_{y} |Q\rangle
\nonumber \\
=& \langle u_{{\cal X}}| \sum_{y \in {\cal Y}} B_y+[W_y,A] |P_{\vec{W}} \rangle
+
\langle u_{{\cal X}}| \sum_{y \in {\cal Y}} W_{y} |Q\rangle
\nonumber \\
\stackrel{(c)}{=}& 
\langle u_{{\cal X}}| \sum_{y \in {\cal Y}} \frac{d}{d\theta}W_{\theta,y}|_{{\theta}=0}|P_{\vec{W}} \rangle
+
\langle u_{{\cal X}}| \sum_{y \in {\cal Y}} W_{y}
\Big|\frac{d}{d\theta}
P_{\vec{W}_{\theta}} 
|_{{\theta}=0}\Big\rangle
\nonumber \\
\stackrel{(d)}{=}&0 ,\Label{27-f} 
\end{align}
where $(a)$, $(b)$, $(c)$, and $(d)$ follow from \eqref{27-a}, \eqref{27-e}, \eqref{27-d}, and \eqref{27-g}, respectively.

Similar to \eqref{9-27-3}, we have
\begin{align}
& \frac{d }{d \theta}
P^{k}[\vec{W}_{\vec{\theta}}]\cdot P_{\vec{W}_\theta} (y_1, \ldots, y_k)|_{{\theta}=0}
\nonumber \\
\stackrel{(a)}{=}
 &
-\langle u_{{\cal X}}| 
A W_{y_k} W_{y_{k-1}} \ldots W_{y_1}  |P_{\vec{W}}\rangle
+
\langle u_{{\cal X}}| 
W_{y_k}
W_{y_{k-1}} \ldots 
W_{y_1}A |P_{\vec{W}}\rangle 
\nonumber \\
&+\langle u_{{\cal X}}|  W_{y_k} W_{y_{k-1}} \ldots W_{y_1}  |Q\rangle
\nonumber \\
\stackrel{(b)}{=} &
-c \langle u_{{\cal X}}| 
 W_{y_k} W_{y_{k-1}} \ldots W_{y_1}  |P_{\vec{W}}\rangle
+
\langle u_{{\cal X}}| 
W_{y_k}
W_{y_{k-1}} \ldots 
W_{y_1}A |P_{\vec{W}}\rangle 
\nonumber \\
&+\langle u_{{\cal X}}| 
 W_{y_k} W_{y_{k-1}} \ldots W_{y_1}  
\Bigg( c '|P_{\vec{W}} \rangle -A |P_{\vec{W}} \rangle
+\bigg(\sum_{y \in {\cal Y}} B_y\bigg) | P_{\vec{W}} \rangle \Bigg)
\nonumber \\
\stackrel{(c)}{=}&
0.\Label{9-27-3B}
\end{align}
Here, $(a)$ follows from a derivation similar to \eqref{9-27-3}.
That is, we need to care about the derivative of $|P_{\vec{W}_\theta} \rangle $.
$(b)$ follows from \eqref{27-a} and \eqref{27-e},
and $(c)$ does from \eqref{27-f}.
So, we obtain (2).

\PF{${\bf (C3)} \Rightarrow {\bf (C1)}$}
Assume {\bf (C3)}.
We define 
$\vec{W}_{\theta,y}'$ in the same way as the proof of Theorem \ref{L27-2}.
So, similar to the proof of Theorem \ref{L27-2},
there exist an invertible map $T_\theta$ on ${\cal V}_{{\cal X}}$
and $(B_{\theta,y})_{y\in {\cal Y}} \in {\cal L}_{2,\vec{W}}$
such that
$T_\theta P_{\vec{W}}=P_{\vec{W}_\theta'}$,
$B_{\theta,y} ({\cal V}^{k_{(P,\vec{W})}}(P) +\Ker P^{k_{\vec{W}}}[\vec{W}] )
\subset \Ker P^{k_{\vec{W}}}[\vec{W}]$
and 
$\vec{W}_{\theta ,y}'= T_\theta^{-1}(W_y+B_{\theta,y})T_\theta$.

Choosing $A$ and $B_y$ in the same way as the proof of Theorem \ref{L27-2}, we have
\begin{align}
\frac{d}{d\theta}\vec{W}_{\theta ,y}|_{\theta=0}=
[W_y,A]+B_y
\Label{9-27-2b}.
\end{align}
Then, in the same way as the proof of Theorem \ref{L27-2}, we obtain 
\eqref{27-2}.
Thus, we obtain {\bf (C1)}.

\section{Proofs of Lemmas \ref{T1-27} and \ref{L1-26-2}}\Label{A3}

To show Lemma \ref{T1-27}, we prepare Lemma \ref{L1-26-3}.

\begin{lemma}\Label{L1-26-3}
Let ${\cal V}_1$ be the direct sum space ${\cal V}_2+{\cal V}_3$
of two vector spaces ${\cal V}_2$ and ${\cal V}_3$
with the condition ${\cal V}_2\cap{\cal V}_3=\{0\}$.
Let ${\cal V}_4$ (${\cal V}_5$) be a subspace of ${\cal V}_2$ (${\cal V}_3$).
Assume that a linear map $\alpha_1$ ($\alpha_2$) from ${\cal V}_6$ to ${\cal V}_2$ (${\cal V}_3$) satisfies that
(1) $\alpha_1({\cal V}_1)\cap {\cal V}_4 = \{0\} $ and (2) $\alpha_2(\Ker \alpha_1)\cap {\cal V}_5 = \{0\} $.
Define $\alpha_3(v):= \alpha_1(v)+\alpha_2(v)$.
Then, $\alpha_3({\cal V}_1)\cap ({\cal V}_4+{\cal V}_5)= \{0\}$.
\end{lemma} 

\begin{proofof}{Lemma \ref{L1-26-3}}
Assume that $\alpha_1(v_1)=v_4$ and $\alpha_2(v_1)=v_5$ 
for $v_1 \in {\cal V}_1$, $v_4 \in {\cal V}_4$, and  $v_5 \in {\cal V}_5$.
Condition (1) implies that $\alpha_1(v_1)=0$.
So, $v_1 \in \Ker \alpha_1$. 
Condition (2) and $\alpha_2(v_1)=v_5$ yield that $\alpha_2(v_1)=v_5=0$, which is the desired statement.
\end{proofof}

\begin{proofof}{Lemma \ref{T1-27}}
Now, we check that the space spanned by $\hat{g}_1, \ldots, \hat{g}_{l_2+l_3} $ has intersection $\{0\}$ with
${\cal N}(({\cal Y}\times {\cal X}^{2})_{\vec{W}})
+{\cal N}_{P_{\vec{W}}}(({\cal Y}\times {\cal X}^{2})_{\vec{W}})
+{\cal N}_2(({\cal Y}\times {\cal X}^{2})_{\vec{W}}))$.
For this purpose, we make preparation.
We choose the matrix $\tilde{A}$ as a diagonal matrix with diagonal entry $f(x)$. 
So, we have $W_y(x|x')(f(x)-f(x'))=[ W_y,\tilde{A}]$. 
We can restrict function $f$ so that $\sum_x f(x)=0$.
Since ${\cal N}(({\cal Y}\times {\cal X}^{2})_{\vec{W}})
= \{ (f(x)-f(x')+c)_{x,x'}  \}$ and $\langle u_{{\cal X}}| \tilde{A} |u_{{\cal X}}\rangle =0$, 
we have
\begin{align}
{\cal L}_{2,\vec{W}}+ \vec{W}_* {\cal N}(({\cal Y}\times {\cal X}^{2})_{\vec{W}})
= 
\{ (c W_y)_{y \in {\cal Y}}\}+
\{(\alpha_{y}(A))_{y \in {\cal Y}} |  \langle u_{{\cal X}}| A|u_{{\cal X}}\rangle =0 \}.
\Label{eq1-26-4}
\end{align}

To prove the above issue, it is sufficient to show that a nonzero element of 
the space spanned by $\hat{g}_1, \ldots, \hat{g}_{l_2+l_3} $ is not contained in 
the space ${\cal N}(({\cal Y}\times {\cal X}^{2})_{\vec{W}})
+{\cal N}_2(({\cal Y}\times {\cal X}^{2})_{\vec{W}}))$.
If a non-zero element is contained in the space,
its matrix components with $y=y_0,y_1$ are given as those of the element of 
the space ${\cal N}(({\cal Y}\times {\cal X}^{2})_{\vec{W}})
+{\cal N}_2(({\cal Y}\times {\cal X}^{2})_{\vec{W}}))$.
To deny this statement, 
we regard $\bar{g}_{j,y_0}$ and  $\bar{g}_{j,y_1}$ as elements of ${\cal G}(\{y_0,y_1\},{\cal X}^2)$.
Then, due to \eqref{eq1-26-4}, 
it is sufficient to show that 
the space spanned by  $\bar{g}_{1,y_0}, \ldots, \bar{g}_{d,y_0}$, $\bar{g}_{1,y_1}, \ldots, \bar{g}_{d^2-d,y_1}$
has intersection $\{0\}$ with the space
$\{(\alpha_{y_0}(A),\alpha_{y_1}(A))|  \langle u_{{\cal X}}| A|u_{{\cal X}}\rangle =0\}$.
To show this statement, we apply Lemma \ref{L1-26-3} to the case when 
${\cal V}_2$ and ${\cal V}_3$ are the set of traceless matrices,
${\cal V}_4$ is the space spanned by  $\bar{g}_{1,y_0}, \ldots, \bar{g}_{d,y_0}$,
${\cal V}_5$ is the space spanned by $\bar{g}_{1,y_1}, \ldots, \bar{g}_{d^2-d,y_1}$,
$\alpha_1$ is the map $\alpha_{y_0}$, and $\alpha_2$ is the map $\alpha_{y_1}$.
Since $\alpha_{y_0}$ is injective on $\{ A \in {\cal M}({\cal V}_{\cal X}) |  A^T |u_{{\cal X}}\rangle=0 \}$ whose dimension is the same as 
that of the image of $\alpha_{y_0}$,
due to the construction of $g_{j,y_0}$,
we find that the map $\alpha_{y_0}$ satisfies the condition for $\alpha_1$.
So, we obtain the desired statement.
\end{proofof}

\begin{proofof}{Lemma \ref{L1-26-2}}
Assume the condition $[W_{y_0},A]=0$.
Then, $A^T$ needs to has common eigenvectors with $W_{y_0}$.
Due to the condition $a^j \neq 0$ for any $j$,
the eigenspace of $A^T$ including $u_{{\cal X}}$ needs to be the whole space.
So, $A^T$ is zero, which implies the condition (1) of Condition E3.

Let $A$ be an element of the kernel of the map $A\mapsto ([W_{y_0},A],[W_{y_1},A])$.
Then, an eigenspace of $A^T$ is spanned by a subset of $\{f_i\}$.
It also is spanned by a subset of $\{f_i'\}$.
To realize both conditions, the eigenspace needs to be the whole space.
So, $A^T$ is zero, which implies the condition (2) of Condition E3.
\end{proofof}

\end{document}